\documentclass[a4paper]{article}
\usepackage{fullpage}
\usepackage{amsmath}
\usepackage{amssymb}
\usepackage{amsthm}
\usepackage{graphicx}
\usepackage{bbm}
\usepackage{enumerate}
\usepackage{microtype}
\usepackage{xcolor}
\usepackage[none]{hyphenat}
\usepackage[colorlinks=true,urlcolor=blue,linkcolor=blue,plainpages=false,pdfpagelabels]{hyperref}
\allowdisplaybreaks

\newtheorem{theorem}{Theorem}[section]
\newtheorem{proposition}[theorem]{Proposition}
\newtheorem{lemma}[theorem]{Lemma}
\newtheorem{corollary}[theorem]{Corollary}
\newtheorem{remark}[theorem]{Remark}
\newtheorem{notation}[theorem]{Notation}
\newtheorem{definition}[theorem]{Definition}

\newcommand{\N}{\mathbb{N}}
\newcommand{\Q}{\mathbb{Q}}
\newcommand{\R}{\mathbb{R}}
\newcommand{\eps}{\varepsilon}
\newcommand{\vp}{\varphi}
\newcommand{\WPsi}{{\widetilde{\Psi}}}
\newcommand{\uu}[1]{\underline{#1}}
\newcommand{\id}{\mathrm{id}}
\DeclareMathOperator*{\argmax}{arg\,max}

\providecommand{\dotdiv}{
  \mathbin{
    \vphantom{+}
    \text{
      \mathsurround=0pt 
      \ooalign{
        \noalign{\kern-.35ex}
        \hidewidth$\smash{\cdot}$\hidewidth\cr 
        \noalign{\kern.35ex}
        $-$\cr 
      }%
    }%
  }%
}

\title{The finitary content of sunny nonexpansive retractions}
\author{Ulrich Kohlenbach$^{a}$ and Andrei Sipo\c s${}^{a,b}$\\[2mm]
\footnotesize ${}^a$Department of Mathematics, Technische Universit\"at Darmstadt,\\
\footnotesize Schlossgartenstrasse 7, 64289 Darmstadt, Germany\\[1mm]
\footnotesize ${}^b$Simion Stoilow Institute of Mathematics of the Romanian Academy,\\
\footnotesize Calea Grivi\c tei 21, 010702 Bucharest, Romania \\[2mm]
\footnotesize E-mails: \{kohlenbach,sipos\}@mathematik.tu-darmstadt.de\\
}
\date{}

\begin{document}

\maketitle

\begin{abstract}
We use techniques of proof mining to extract a uniform rate of metastability (in the sense of Tao) for the strong convergence of approximants to fixed points of uniformly continuous 
pseudocontractive mappings in Banach spaces which are uniformly convex and uniformly smooth, i.e. a slightly restricted form of the classical result of Reich. This is made possible by the 
existence of a modulus of uniqueness specific to uniformly convex Banach spaces and by the arithmetization of the use of the limit superior. The metastable convergence can thus be proved 
in a system which has the same provably total functions as first-order arithmetic and therefore one may interpret the resulting proof in G\"odel's system $T$ of higher-type functionals. 
The witness so obtained is then majorized (in the sense of Howard) in order to produce the final bound, which is shown to be definable in the subsystem $T_1$. This piece of information is 
further used to obtain rates of metastability to results which were previously only analyzed from the point of view of proof mining in the context of Hilbert spaces, i.e. 
the convergence of the iterative schemas of Halpern and Bruck.


\end{abstract}

\section{Introduction}\label{sec:intro}

Let $(X,\|\cdot \|)$ be a real Banach space, $C\subseteq X$ be a 
nonempty bounded 
closed convex subset and $T:C\to C$ be a nonexpansive mapping. For 
$t\in (0,1)$ and $x\in C,$ let $x_t$ be the unique fixed point of 
the strict contraction 
\[ T_t: C\to C, \ \ T_t(y):=tT(y)+(1-t)x. \]

In 1967, 
Browder \cite{Bro67A} and Halpern \cite{Hal67} independently proved in the 
case where $X$ is a Hilbert space that 
for $t\to 1,$ the path $(x_t)$ strongly converges and its limit is the fixed point of $T$ which is closest to $x,$ i.e. $Px,$ where 
$P:C\to Fix(T)$ is the metric projection onto $Fix(T).$ 
Both proofs for the strong convergence 
do not readily generalize even to the class of $L^p$ spaces (other than $L^2$). 

That the strong 
convergence does hold in this case was  
finally shown in 
1980, when Reich established in the celebrated paper \cite{Rei80} that it actually holds in any uniformly smooth space. Moreover, 
Reich showed that the limit is $Qx,$ where $Q$ is 
the unique sunny nonexpansive retraction 
$Q:C\to Fix(T)$. This result has subsequently been extended in many ways including the context of families of operators 
\cite{AleynerReich05,AleynerReich05a,AleynerReich07}. 

The significance of Reich's theorem is twofold:
\begin{itemize}
\item 
It provides for the first time an algorithmic approach to the 
construction of sunny nonexpansive retractions. This aspect is highlighted 
e.g. in \cite{Benavides,AleynerReich05}.
\item
Many important iterative algorithms in nonlinear analysis are shown 
to be strongly convergent by proving that they asymptotically approach 
$(x_{t_n})$ 
(for some suitable sequence $(t_n)\subseteq (0,1)$ converging to $1$).
\end{itemize}
We start discussing the first item in more detail.
Nonexpansive retractions were first considered by Bruck in \cite{Bru70}, 
who showed -- using Zorn's lemma -- that $Fix(T)$ is a nonexpansive retract 
of $C,$ whenever $X$ is a real reflexive strictly convex Banach space. 
This result was generalized further in \cite{Bru73}, in particular, 
to reflexive Banach spaces which have the conditional fixed point property for 
nonexpansive mappings which e.g. includes all uniformly smooth spaces.
Since metric projections onto closed convex subsets are nonexpansive only 
in Hilbert spaces, nonexpansive retractions are, already for $L^p$ spaces (again, other than $L^2$), very different from 
metric projections and may not 
even exist although the metric projection does. For example, Bruck showed in \cite{Bru74}
that no real Banach space $X$ with $\dim X\ge 3$ has a 
bounded smooth subset 
$E\subset X$ with nonempty interior which 
is the range of a nonexpansive retraction 
$Q:X\to E$ unless $X$ is a Hilbert space.

Retractions $Q:C\to E\subseteq C$ are 
called {\it sunny} if the property 
\[  \forall x \in C\,\forall t \geq 0 \ (Qx+t(x-Qx)\in C\to 
Q(Qx+t(x-Qx))=Qx) \] 
holds. In smooth Banach spaces, for a retraction $Q$ to be nonexpansive 
and sunny it is necessary and sufficient for the variational inequality  (where $j$ denotes the single-valued normalized duality map) 
\[  \forall x\in C\,\forall y \in E \ (\langle x-Qx,j(y-Qx)\rangle \leq 0)\] 
to hold, which in Hilbert spaces characterizes the metric projection. Thus, the relevance of sunny nonexpansive retractions is that they are in many respects  the right substitute for the metric projection outside Hilbert spaces.
From this characterization it follows that there is at most one 
sunny nonexpansive retraction $Q:C\to E$ in smooth spaces 
(in \cite{Bru73A}, Bruck used the term `nonexpansive projection' instead 
of the nowadays common name `sunny nonexpansive retraction'). 
If $X$ is even 
uniformly smooth and strictly convex and $E=Fix(T)$ is the fixed point 
set of a nonexpansive mapping $T:C\to C,$ then the 
unique sunny nonexpansive retraction $Q:C\to Fix(T)$ necessarily exists  \cite{Bru73A}.
 Bruck's proof is, however, highly nonconstructive. Reich's theorem establishes 
that the sunny nonexpansive retraction can be obtained as the limit of 
objects $x_t$ which are constructively available (since Banach's fixed 
point theorem is constructive). Our logical analysis of the proofs due to Morales
of Reich's theorem implies that the pointwise existence of sunny 
nonexpansive retracts can be carried out in a logically fairly weak 
formal system (see Remark \ref{rem.6.5}) which is of foundational interest.

As stated in the second item above, 
the great relevance of Reich's theorem 
for algorithmic purposes can also be seen from the fact that 
it implies the strong convergence of important iterative algorithms: in \cite{Hal67}, 
the so-called Halpern iteration (starting from some $x_0\in C$ and using 
$u\in C$ as anchor)
\[ x_{n+1}:=\lambda_{n+1} u+(1-\lambda_{n+1})Tx_n \]
is considered for $(\lambda_n)\subset [0,1]$ and -- in Hilbert spaces -- 
shown to converge to $Pu$ for the metric projection $P:C\to Fix(T)$ 
under very restrictive conditions on $(\lambda_n).$ In a milestone
paper \cite{Witt92}, Wittmann generalized this to much more general sequences 
$(\lambda_n)$ including for the first time the case $\lambda_n:=1/(n+1).$ 
If $T$ is linear and $u=x_0$, then $x_n$ coincides (for this choice of 
$(\lambda_n)$) with the ergodic average 
$\frac{1}{n+1} \sum^n_{i=0} T^ix_0$ and so Wittmann's theorem is a nonlinear generalization of the classical von Neumann mean ergodic theorem, while remaining strongly convergent (without 
linearity, the usual ergodic averages are known to converge only weakly by 
results due to \cite{Bai75} and \cite{GenLin}). 
In \cite{ShiTak97}, Wittmann's theorem 
is generalized to uniformly smooth Banach spaces by reducing the strong 
convergence of $(x_n)$ to that of $(x_t)$ and then applying Reich's theorem 
(in fact, \cite{ShiTak97} considers a somewhat larger class of spaces). 
For Halpern's more restrictive sequences $(\lambda_n),$ this had already been shown 
in \cite{Rei80}.

Reich \cite{Rei80} established his theorem not only for nonexpansive mappings 
but even for set-valued accretive operators satisfying the range 
condition which, in particular, covers 
the important class of continuous pseudocontractions, introduced by Browder \cite{Bro67}, 
which extend the class of nonexpansive mappings and which play a crucial role
in the abstract formulation of Cauchy problems. For pseudocontractions one can no longer 
 use the Halpern iterative schema but has to apply a more complicated schema due 
to Bruck \cite{Bru74A}
\[
 x_{n+1}:=\left(1-\lambda_n\right)x_n+\lambda_n Tx_n-
\lambda_n\theta_n
\left(x_n-x_1\right) \] 
for suitable sequences $(\lambda_n)$, $(\theta_n)$ in $[0,1].$ 
In \cite{Chi}, it is shown that for Lipschitzian pseudocontractions 
(a class which still strictly generalizes the class of nonexpansive 
mappings and which contains the class of strict pseudocontractions due to 
\cite{BroPet67}) the strong convergence of the Bruck iteration schema can be 
shown using the strong convergence of $(x_t),$ i.e. again by reduction to Reich's theorem.

Furthermore, recently, in \cite{Aoyama}, a Halpern-type variant of the 
famous proximal point algorithm was shown to strongly converge 
by a similar reduction.

These and many other results point to the paramount significance of this result of Reich.
In this paper, we give for the first time  a quantitative account of it.
From results of Neumann \cite{Neu15} on the Halpern 
iteration and the aforementioned connection with the convergence of 
$(x_t)$ (which was treated quantitatively in \cite{KohLeu12}), it follows 
that even for the case of Hilbert spaces, in fact, already for $X:=\R$ and 
$C:=[0,1],$ there are simple computable mappings $T:C\to C$ for which 
$(x_n):=(x_{1-\frac{1}{n+1}})$ with the anchor point $x:=0$ does not have a computable rate 
of convergence. In this situation, the next best thing one can hope for 
is an effective so-called rate of metastability -- in the sense of Terence Tao 
\cite{Tao08,Tao08A}, the name having been suggested to him by Jennifer Chayes --
i.e. a function $\Theta:\N\times \N^{\N}\to \N$ such that 
\[ (*) \ \forall k\in\N\,\forall g\in\N^{\N} \,\exists N\le 
\Theta(k,g)\,
\forall n,m\in [N,N+g(N)] \ \left(\| x_n-x_m\| <\frac{1}{k+1}\right), \] 
where $[N,N+g(N)]:=\{ N,N+1,N+2,\ldots, N+g(N)\},$ whose complexity 
reflects the 
computational content of the original convergence proof from which it is 
extractable by proof-theoretic methods (see \cite{Koh08}). 
Note that $(*)$ provides a quantitative form of 
\[   \forall k\in\N \,\forall g\in\N^{\N} \,\exists N\in \N \,
\forall n,m\in [N,N+g(N)] \ \left(\| x_n-x_m\| <\frac{1}{k+1}\right), \] which, 
noneffectively, is equivalent to the ordinary Cauchy property of 
$(x_n).$ In proof theory, the metastable version of the original 
Cauchy statement is known as the Kreisel no-counterexample 
interpretation \cite{Kreisel(51),Kreisel(52)}. General  
so-called logical metatheorems due to 
\cite{Koh05,GerKoh08,Koh08,Engracia(diss),KohLeu12,GunKoh16} 
guarantee the extractability of explicit effective bounds, in particular of 
rates of metastability, for large classes of proofs and provide 
algorithms for their actual extraction from a given proof based on modern 
variants and extensions of G\"odel's \cite{Goe58} famous functional 
(`Dialectica') interpretation. Moreover, these bounds 
only depend on $X$, $C$ and $T$ via `majorizing' data (such as moduli of 
smoothness on $X$ or of uniform continuity of $T$ and norm bounds on the 
elements of $C$). These developments are all part of the research program of `proof mining', that aims to apply these logical techniques to proofs in a broad range of areas of mainstream mathematics, such as nonlinear analysis, convex optimization, commutative algebra, ergodic theory or topological dynamics; the standard introduction to the field is \cite{Koh08}, while more recent surveys are \cite{Koh17,Koh18}. 

For the Hilbert space case of the problem at hand, such $\Theta$'s of low primitive recursive complexity have already been extracted both for the Browder-Halpern theorem and for Wittmann's theorem in \cite{Koh11}, and an alternative way of using proof mining to derive these and related results was recently explored in \cite{FerLeuPin}. 

However, a quantitative analysis of Reich's generalization to Banach spaces had been a major challenge of the `proof mining' paradigm for about ten years. The present paper, which for the first time succeeds in achieving such an analysis, is the technically most complex extraction of a metastability bound for a strong convergence theorem in 
analysis which has ever been carried out. The enormous complexity of the final bound reflects the profound 
combinatorial and computational content of Reich's deep theorem.

More specifically, in the present paper, we extract for the first time a rate of metastability 
for  the convergence of $(x_t)$ for uniformly continuous 
pseudocontractions within the class of Banach spaces which are 
uniformly smooth and uniformly convex (which covers all $L^p$ spaces for 
$1<p<\infty$). Using quantitative results extracted 
already in \cite{KohLeu12}, this also gives the first 
explicit rate of metastability 
for the extension (due to \cite{ShiTak97}) of Wittmann's theorem to this 
class of spaces as well as, using quantitative results from \cite{KorKoh14}, 
the first rate of metastability for Bruck's iteration for this class. All 
previous results only considered the class of Hilbert spaces (or geodesic 
generalizations of Hilbert spaces such as CAT(0) spaces \cite{KohLeu12A} or 
CAT($\kappa$) spaces for $\kappa\geq0$ \cite{LeuNic16}). As predicted by general 
logical metatheorems from \cite{Koh08,KohLeu12}, the rate of metastability 
(in the case where $t_n:=1-\frac{1}{n+1}$) 
only depends (in addition to $\varepsilon$ and $g$) on a norm bound $b$ on the 
elements in $C$, on moduli $\tau$, $\eta$ of uniform smoothness and convexity, respectively,
of $X$ and on a modulus $\theta$ of uniform continuity of $T$.

Our extraction of $\Theta$ analyzes the proof of Reich's theorem given in 1990 by Morales 
\cite{Mor90}. This proof uses that the continuous convex function
\[ F(z):=\limsup_{n\to\infty} \| x_{t_n}-z\|^2, \] 
where $(t_n)$ is a sequence in $(0,1)$ which converges to $1,$ attains its 
infimum on the closed convex bounded set $C$ since $X$ is reflexive, being uniformly smooth. (Reich's original proof \cite{Rei80} produced the operator $F$ as the limit of a subsequence, which was shown to be well-defined in \cite{Rei80A}; later developments of the idea, even to this day, generally use a simplification of this by applying a Banach limit to the original sequence -- see, e.g., \cite{BruRei81,BruKirRei82,TakUed84}; to our knowledge, the above definition -- lifted from the theory of asymptotic centers \cite{Ede72} -- was first used by Morales in this context and afterwards picked up by few other authors.)
The proof then continues by forming the set of all points on which $F$ attains 
its infimum, showing that this set is invariant under the action of (the resolvent of) $T$ and thus (the resolvent, and therefore) $T$ 
has a fixed point in this set. (The detour via the resolvent is not needed for nonexpansive mappings.) In the 
deductive framework to which the known proof-theoretic bound extraction 
methods apply, it is not clear how to define $F$ as an object given as we 
do not have a term which assigns to a bounded sequence of reals its 
$\limsup$ (in technical terms this is due to the fact the functional 
interpretation of having such a term has no solution by majorizable 
functionals; only if $X$ would be assumed as separable -- which we have to 
avoid, however, 
for general reasons discussed in \cite{Koh08} as this prevents the extraction 
of uniform bounds -- then using the continuity of $F$ it would be enough  
to define $F$ on a dense sequence and this could be done in our setting). 
So we aim to replace the use of $F$ as an object by 
\[ (**)\ 
\forall z\in C\,\exists a\in \R^+\, (a= \limsup_{n\to\infty} \| x_{t_n}-z\|^2), \]
where `$a=\limsup_{n\to\infty} \| x_{t_n}-z\|^2$' is logically complex, namely it is a 
so-called $\Pi^0_3$ statement.

This makes it difficult to formalize the above arguments in a setting 
which only allows one to use $(**).$  
That is why we add the additional assumption that $X$ is a {\bf uniformly 
convex} Banach space which yields that $F$ is a uniformly convex function.
This is usually used to prove that asymptotic centers are unique in this class of spaces, and
we show that one can construct (by way of Proposition~\ref{l3-a}) a modulus of uniqueness for the infimum 
problem stating that -- given $\varepsilon>0$ -- there is a $\delta>0$ such that 
$\delta$-approximate infima points are $\varepsilon$-close to each other (for more details, see e.g. its use in Claims 2 and 3 of the proof in Section~\ref{section-limsup}). 
It is then sufficient to consider only $\delta$-infima points instead 
of actual infima points. The resulting proof can then be shown to be formalizable 
with the use of arithmetical comprehension which already guarantees the 
extractability of a rate of metastability which is definable in the 
calculus $T+B_{0,1},$ where $T$ is the system of the Hilbert-G\"odel \cite{Hil26,Goe58} 
primitive recursive 
functionals of finite type and $B_{0,1}$ is the schema of Spector's 
\cite{Spec62} bar 
recursion (of lowest type). We then show that the use of real limsup's 
can be replaced -- using a process of arithmetization, see 
\cite{Koh97A} and Remark~\ref{rem-ar} -- by that of $\varepsilon$-limsup's whose existence can 
be shown using just induction (more precisely, using $\Pi^0_2$-induction, to which it is equivalent and which -- by Parsons \cite{Par72} -- has a solution in 
the fragment $T_1$ of $T$). 
 
Since the existence of $\delta$-infima of $F$ also requires only induction, 
it follows from this that one gets a rate of metastability which is 
primitive recursive in the extended sense of $T.$ The analysis 
of the $\delta$-infima argument shows that $T_2$ suffices. When the details 
of the extraction are all carried out, it turns out that for the particular 
instances of that argument used, actually $T_1$ suffices, which, therefore, 
is the complexity of our final bound. The statement $(*)$ with this explicit 
bound provides a finitary version (in the sense of Tao) of the theorem that 
$(x_t)$ converges to the sunny nonexpansive retraction $Qx$ of $x$ 
(and so, in particular, also of the existence of $Q$ itself) 
since the latter can 
be derived from $(*)$ by an elementary proof. In particular, it follows 
that only a single instance of $\Pi^0_1$-comprehension is needed (or, as seen from the viewpoint of constructive mathematics and in the presence of 
$\Pi^0_1$-AC$^{0,0}$ choice for numbers, only the $\Sigma^0_1$ law of excluded 
middle) to derive the theorem. We believe that our analysis exhibits the 
explicit numerical content of the existence proof for $Q.$ Only future research
will show whether the complexity class $T_1$ is the best possible or whether an 
ordinary primitive recursive rate $\Theta\in T_0$ can be achieved (or even 
whether a close examination of our bound might show that it can already be 
defined in $T_0$, see Remark \ref{complexity}).

The next section introduces the preliminary notions used to discuss and prove our result, namely on uniformly convex and uniformly smooth spaces, and on nonexpansive retractions and pseudocontractions. Highlights include the modulus of convexity for the squared norm of a uniformly convex space -- which has as an immediate consequence the uniform convexity of the function $F$ discussed above -- as well as the introduction of the resolvent $g_T$ of a continuous pseudocontraction $T$ that allows one to use nonexpansiveness arguments as needed. Section~\ref{section-limsup} details the way to an intermediate proof of the main result where the use of $F$ as an operator has been eliminated and only $\eps$-infima of it are needed, which are made useful by means of the modulus of uniqueness. In Section~\ref{section-approx} even this use of $F$ in the form of pointwise limsup's is removed, as they are replaced with approximate limsup's. Some care must be taken to ensure that approximate limsup's may be shown to exist using just induction (Proposition~\ref{exlimsup}) and that they are useful for our purposes (Lemma~\ref{lpe}). Finally, in Section~\ref{sec:proof-mining}, the higher-order portions of the witness extraction are carried out, yielding a highly complex, though structured, realizer. In Section~\ref{sec:metastab} this realizer is progressively majorized in order to obtain our goal, namely a rate of metastability. It is argued there both that this final bound $\Theta$ is expressible in $T_1$ and that the metastability statement is a true finitization (again in the sense of T. Tao) of the full form of the original strong convergence statement. Playing the role of an epilogue, Section~\ref{sec:apps} presents two completions by means of our result of proof mining analyses which had only been carried partially 
(in the sense that a rate of metastability was produced assuming such a 
rate for $(x_{t_n})$ be given which so far was known only in the Hilbert 
space case), 
namely the strong convergence of the iterations of B. Halpern and R. E. Bruck.

\section{Preliminaries}\label{sec:prelim}

\subsection{Classes of Banach spaces}

\subsubsection{Uniformly convex spaces}

\begin{definition}[{cf. \cite{Cla36,Day44}}]
Let $X$ be a Banach space. We call the function $\delta_X : (0,2] \to \R$, defined, for all $\eps \in (0,2]$, by:
$$\delta_X(\eps):= \inf \left\{ 1- \left\| \frac{x+y}2 \right\| \bigm| \|x\|=\|y\|=1, \|x-y\| \geq \eps \right\}$$
``the'' {\bf modulus of convexity} of $X$.
\end{definition}

The following result shows that this modulus can be obtained in a less strict way.

\begin{proposition}[{\cite[p. 60]{LinTza79}}]
Let $X$ be a Banach space. Then, for all $\eps \in (0,2]$,
\begin{align*}
\delta_X(\eps)&=  \inf \left\{ 1- \left\| \frac{x+y}2 \right\| \bigm| \|x\|\leq 1, \|y\|\leq 1, \|x-y\| \geq \eps \right\}.
\end{align*}
\end{proposition}

\begin{corollary}
Let $X$ be a Banach space. TFAE:
\begin{enumerate}[(a)]
\item for all $\eps \in (0,2]$, $\delta_X(\eps)>0$.
\item there is an $\eta : (0,2] \to (0,1]$ (called ``a'' {\bf modulus of convexity}) such that for all $\eps \in (0,2]$ and all $x,y \in X$ with $\|x\|\leq 1$, $\|y\| \leq 1$ and $\|x-y\| \geq \eps$ one has that
$$\left\|\frac{x+y}2\right\| \leq 1-\eta(\eps).$$
\end{enumerate}
(One can, obviously, for the implication ``(a) $\Rightarrow$ (b)'', put, for all $\eps$, $\eta(\eps):=\delta_X(\eps)$.) In this case, $X$ is called {\bf uniformly convex}.
\end{corollary}

The following is an application of a recent proof mining result of Ba\v{c}\'ak and the first author, specifically \cite[Proposition~3.2]{BacKoh18}, itself a quantitative version of a theorem of Z\u alinescu \cite[Theorem~4.1]{Zal83}. We remark that a similar kind of result (i.e. with a different modulus) may be obtained by adapting an argument from \cite[Section~2]{XuRoa91} to work with $\eta$ instead of $\delta_X$. The non-quantitative version may also be found in the statement of \cite[Theorem~2]{Xu91}, but the proof given there is highly non-constructive.

\begin{proposition}\label{l3-a}
Let $X$ be a uniformly convex Banach space having $\eta$ as a modulus and let $b \geq \frac12$. Put, for all $\eps \in (0,2]$,
$$\psi_{b,\eta}(\eps):= \min \left( \frac{\left(\min\left(\frac\eps2, \frac{\eps^2}{72b}\eta^2\left(\frac\eps{2b}\right)\right)\right)^2}4, \frac{\eps^2}{48}\eta^2\left(\frac{\eps}{2b}\right) \right).$$
Then, for all $\eps \in (0,2]$:
\begin{enumerate}[(a)]
\item $\psi_{b,\eta}(\eps) > 0$.
\item for all $x,y \in X$ with $\|x\|\leq b$, $\|y\| \leq b$, $\|x-y\| \geq \eps$, we have that
$$\left\| \frac{x+y}2 \right\|^2 + \psi_{b,\eta}(\eps) \leq \frac12\|x\|^2 + \frac12\|y\|^2.$$
\end{enumerate}
\end{proposition}

\begin{proof} We may assume that $x,y\not= 0.$
We seek to apply \cite[Proposition~3.2]{BacKoh18}. We need, then, only to pass from $x$ to $\frac{x}{\|x\|}$ and from $y$ to $\frac{y}{\|y\|}$ and then to put $r:=b$, $\alpha:=\|x\|$, $\beta:=\|y\|$ and $\Phi$ to be the squaring function. To obtain the conclusion, one has to verify that, for an arbitrary $r>0$, the squaring function has on the interval $[0,r]$ the function $\eps \mapsto \frac{\eps^2}4$ as a modulus of uniform convexity, $\eps \mapsto \frac{\eps}{2r}$ as a modulus of uniform continuity and $\eps \mapsto \eps^2$ as a modulus of uniform increasingness.
\end{proof}

\subsubsection{Smooth and uniformly smooth spaces}

\begin{definition}
Let $X$ be a Banach space. We define the {\bf normalized duality mapping of $X$} to be the map $J :X\to 2^{X^*}$, defined, for all $x \in X$, by
$$ J(x) := \{x^*\in X^* \mid  x^*(x)=\|x\|^2,\ \|x^*\|=\|x\| \}.$$
\end{definition}

A Banach space $X$ is called {\it smooth} if for any $x \in X$ with $\|x\|=1$, we have that for any $y \in X$ with $\|y\|=1$, the limit
\begin{equation}\label{limh}
\lim_{h \to 0} \frac{\|x+hy\|-\|x\|}{h}
\end{equation}
exists. This condition has been proven to be equivalent to the fact that the normalized duality mapping of the space, $J: X \to 2^{X^*}$, is single-valued -- and we shall denote its unique section by $j: X \to X^*$. Therefore, for all $x \in E$, $j(x)(x)=\|x\|^2$ and $\|j(x)\|=\|x\|$. Hilbert spaces are smooth, and clearly $j(x)(y)$ is then simply $\langle y,x \rangle$, for any $x,y$ in the space. Because of this, we may consider the $j$ to be a generalized variant of the inner product, sharing some of its nice properties. We shall generally denote, for all spaces $X$, all $x^* \in X^*$ and $y \in X$, $x^*(y)$ by $\langle y,x^* \rangle$. In addition, the homogeneity of $j$ -- i.e. that for all $x \in X$ and $t \in \R$, $j(tx)=tj(x)$ -- follows immediately from the definition of the duality mapping.

\begin{remark}
These notions of smoothness were introduced in {\rm \cite{Day44}}, under the name of {\em flattening}.
\end{remark}

\begin{lemma}[{cf. \cite[Lemma 1]{Pet70}}]\label{jl}
Let $X$ be a smooth space and $x,y \in X$. Then
$$\|x+y\|^2 \leq \|x\|^2 + 2\langle y, j(x+y) \rangle$$
\end{lemma}

\begin{proof}
We have that
\begin{align*}
\|x+y\|^2 &= \langle x+y, j(x+y) \rangle \\
&= \langle x, j(x+y) \rangle + \langle y, j(x+y) \rangle \\
&\leq \|x+y\|\|x|| + \langle y, j(x+y) \rangle \\
&\leq \frac12(\|x\|^2 + \|x+y\|^2) + \langle y, j(x+y) \rangle,
\end{align*}
from which the conclusion follows.
\end{proof}

\begin{definition}[{\cite[Definition 1.e.1.(ii)]{LinTza79}}]
Let $X$ be a Banach space. We call the function $\rho_X : (0,\infty) \to \R$, defined, for all $t>0$, by
$$\rho_X(t):= \sup \left\{ \frac{\|x+y\| + \|x-y\|}2 - 1 \bigm| \|x\|=1, \|y\|=t\right\},$$
``the'' {\bf modulus of smoothness} of $X$. We remark that for all $t$, $0 \leq \rho_X(t) \leq t$.
\end{definition}

The following characterization is immediate.

\begin{proposition}
Let $X$ be a Banach space. TFAE:
\begin{enumerate}[(a)]
\item $\lim\limits_{t \to 0} \frac{\rho_X(t)}t = 0$.
\item there is a $\tau : (0, \infty) \to (0, \infty)$  (called ``a'' {\bf modulus of smoothness}) such that for all $\eps > 0$ and all $x,y \in X$ with $\|x\|= 1$, $\|y\| \leq \tau(\eps)$ one has that
$$\|x+y\| + \|x-y\| \leq 2 + \eps\|y\|.$$
\end{enumerate}
In this case, $X$ is called {\bf uniformly smooth}.
\end{proposition}

\begin{remark}
A uniformly smooth space is smooth, and this condition is equivalent to the limit in \eqref{limh} being attained uniformly in the pair of variables $(x,y)$.
\end{remark}

\begin{remark}
Unlike in the case of convexity, ``the'' modulus of smoothness is not ``a'' modulus of smoothness.
\end{remark}

\begin{proposition}[{cf. \cite[Proposition 2.5]{KohLeu12}}]
Let $X$ be a uniformly smooth Banach space with modulus $\tau$. Put, for all $b > 0$ and $\eps > 0$,
$$r_1(\eps):=\min(\eps,2), \qquad r_2(b):= \max(b,1), \qquad \omega_\tau(b,\eps) := \frac{r_1(\eps)^2}{12r_2(b)} \cdot \tau\left(\frac{r_1(\eps)}{2r_2(b)}\right).$$
Then for all $b >0$, $\eps >0$ and all $x, y \in X$ with $\|x\| \leq b$ and $\|y \| \leq b$, if $\|x-y\| \leq \omega_\tau(b,\eps)$ then $\|j(x)-j(y)\| \leq \eps$.
\end{proposition}

In the PhD thesis of B\'enilan \cite[p. 0.5, Proposition 0.3]{Ben72}, it is shown that the norm-to-norm uniform continuity on bounded subsets of an arbitrary duality selection mapping is in fact equivalent to uniform smoothness. A more recent proof which uses ideas due to Giles \cite{Gil67} may be found in \cite[Appendix A]{Kor15}.

\subsection{Classes of mappings}

In this section, we fix a smooth Banach space $X$ and $C \subseteq X$ a closed, convex, nonempty subset.

\subsubsection{Nonexpansive mappings and sunny nonexpansive retractions}

\begin{definition}
A map $Q:C \to X$ is called {\bf nonexpansive} if for all $x, y \in C$, $\|Qx-Qy\|\leq\|x-y\|$.
\end{definition}

Let $E \subseteq C$ be nonempty.

\begin{definition}
A map $Q:C\to E$ is called a {\bf retraction} if for all $x \in E$, $Qx=x$.
\end{definition}

\begin{definition}
A retraction $Q:C \to E$ is called {\bf sunny} if for all $x \in C$ and $t \geq 0$,
$$Q(Qx+t(x-Qx))=Qx, \ \mbox{if} \ Qx+t(x-Qx)\in C.$$
\end{definition}

\begin{proposition}[{\cite[Lemma 1.13.1]{GoeRei84}}]\label{char-sunny}
Let $Q: C\to E$ be a retraction. Then $Q$ is sunny and nonexpansive if and only if for all $x\in C$ and $y \in E$,
$$\langle x-Qx,j(y-Qx)\rangle \leq 0.$$
\end{proposition}

\begin{proposition}
There is at most one sunny nonexpansive retraction from $C$ to $E$.
\end{proposition}

\begin{proof}
Let $Q_1$ and $Q_2$ be two such retractions. Let $x\in C$. It follows that 
$$\langle x-Q_1x, j(Q_2x-Q_1x)\rangle \leq 0$$
and
$$\langle x-Q_2x, j(Q_1x-Q_2x)\rangle \leq 0.$$
Using the homogeneity of $j$ and then summing up, it follows that $\|Q_2x-Q_1x\|^2 \leq 0$ and therefore $Q_1x=Q_2x$.
\end{proof}

\subsubsection{Pseudocontractions}

\begin{definition}
Let $T: C \to C$. We call a function $\theta: (0,\infty) \to (0, \infty)$ a {\bf modulus of continuity} for $T$ if for all $\eps > 0$ and $x, y \in C$ with $\|x-y\|\leq \theta(\eps)$, we have that $\|Tx-Ty\|\leq\eps$.
\end{definition}

\begin{remark}
A map $T: C \to C$ has a modulus of continuity iff it is uniformly continuous.
\end{remark}

\begin{definition}[{\cite[Definition 1]{Bro67}}]
A map $T:C \to C$ is called a {\bf pseudocontraction} if for all $x, y \in C$ and $t >0$, we have that
\begin{equation}
t\|x-y\| \leq \|(t+1)(x-y) - (Tx-Ty)\|.\label{def-psc}
\end{equation}
\end{definition}

\begin{proposition}
Any nonexpansive map is a pseudocontraction.
\end{proposition}

\begin{proof}
Let $T: C \to C$ be nonexpansive. Let $x,y \in C$ and $t > 0$. We have that
\begin{align*}
\left(1+ \frac1t\right)\|x - y\| &\leq \left\| \left(1+\frac1t\right)(x-y) - \frac1t(Tx-Ty) \right\| + \frac1t\|Tx-Ty\| \\
&\leq \left\| \left(1+\frac1t\right)(x-y) - \frac1t(Tx-Ty)\right\| + \frac1t\|x-y\|,
\end{align*}
so
$$\|x-y\| \leq \left\| \left(1+\frac1t\right)(x-y) - \frac1t(Tx-Ty)\right\|.$$
Multiplying by $t$, we obtain our conclusion.
\end{proof}

We have the following equivalence.

\begin{proposition}[{\cite[Proposition 1]{Bro67}}]
Let $T: C\to C$. Then $T$ is a pseudocontraction if and only if for all $x,y \in C$,
$$\langle Tx-Ty,j(x-y) \rangle \leq \|x-y\|^2.$$
\end{proposition}

\begin{definition}[{cf. \cite[(2.9)]{Gus67}}]
Let $k \in (0,1)$. A map $T: C \to C$ is called a {\bf $k$-strong pseudocontraction} if for all $x,y \in C$,
$$\langle Tx-Ty,j(x-y) \rangle \leq k\|x-y\|^2.$$
\end{definition}

\begin{proposition}\label{sps1}
Let $T: C \to C$ be a continuous pseudocontraction, $k \in (0,1)$ and $u \in C$. Define the map $U: C \to C$, by putting, for all $x \in C$, $Ux:=kTx+(1-k)u$. Then $U$ is a continuous $k$-strong pseudocontraction.
\end{proposition}

\begin{proof}
We have that for all $x$, $Tx=\frac1k Ux - \frac{1-k}k u$, so for all $x,y$,
$$\left\langle \frac1k Ux - \frac1k Uy, j(x-y) \right\rangle \leq \|x-y\|^2,$$
from which our conclusion follows.
\end{proof}

\begin{proposition}\label{sps2}
Let $k \in (0,1)$ and $T: C \to C$ be a continuous $k$-strong pseudocontraction. Then $T$ has a unique fixed point.
\end{proposition}

\begin{proof}
If $x$ and $y$ are fixed points of $T$, $\|x-y\|^2\leq k\|x-y\|^2$, so $x=y$. The existence of a fixed point follows from \cite[Proposition 3]{Mar73} and the convexity of $C$.
\end{proof}

\begin{definition}
If $T: C \to C$ is a pseudocontraction, we define the map $f_T: C \to X$, for all $x \in C$, by $f_T(x):=2x-Tx$.
\end{definition}

\begin{proposition}
Let $T: C \to C$ be a continuous pseudocontraction. Then for all $y \in C$ there is a unique $x \in C$ such that $f_T(x)=y$.
\end{proposition}

\begin{proof}
Let $y\in C$. Define the map $U: C \to C$, for all $z \in C$, by $Uz:=\frac{Tz+y}2$. Then, by Proposition~\ref{sps1}, $U$ is a continuous $\frac12$-strong pseudocontraction. Since we have that for all $x \in C$, $f_T(x)=y$ iff $Ux=x$, the conclusion follows by applying Proposition~\ref{sps2}.
\end{proof}

\begin{definition}
If $T: C \to C$ is a continuous pseudocontraction, we define the map $g_T: C \to C$ by putting, for all $y \in C$, $g_T(y)$ to be the unique $x \in C$ such that $f_T(x)=y$.
\end{definition}

\begin{notation}
For any function $\theta : (0,\infty) \to (0, \infty)$ and for any $\eps>0$, put $\tilde{\theta}(\eps):=\min\left(\frac\eps4, \theta\left(\frac\eps2\right)\right)$.
\end{notation}

\begin{proposition}
Let $T: C \to C$ be a continuous pseudocontraction. Then:
\begin{enumerate}[(i)]
\item for all $y \in C$, $f_T(g_T(y))=y$;
\item $g_T$ is nonexpansive;
\item for all $x \in C$, $\|x-g_Tx\| \leq \|x-Tx\|$;
\item $g_T$ and $T$ have the same fixed points;
\item if $T$ is uniformly continuous with modulus $\theta$, then for all $x \in C$ and all $\eps>0$, with $\|x-g_Tx\| \leq \tilde{\theta}(\eps)$, we have that $\|x-Tx\| \leq \eps$.
\end{enumerate}
\end{proposition}

\begin{proof}
\begin{enumerate}[(i)]
\item \label{i1} Immediately, from the definition of $g_T$.
\item Let $x, y \in C$ and apply \eqref{def-psc} for $x \mapsto g_T(x)$, $y \mapsto g_T(y)$ and $t \mapsto 1$ to obtain -- using (\ref{i1}) -- that
$$\|g_Tx-g_Ty\| \leq \|f_Tg_Tx - f_Tg_Ty\| = \|x-y\|.$$
\item \label{i3} Let $x \in C$ and apply \eqref{def-psc} for $x \mapsto x$, $y \mapsto g_T(x)$ and $t \mapsto 1$ to obtain -- again, using (\ref{i1}) --  that
$$\|x-g_Tx\| \leq \|f_Tx - f_Tg_Tx\| = \|f_Tx -x \| = \|x-Tx\|.$$
\item \label{i4} One direction follows from (\ref{i3}). For the other, let $p \in C$ be such that $g_Tp=p$. Then $p=f_Tg_Tp=f_Tp=2p-Tp$, so $p$ is a fixed point of $T$.
\item What follows is a quantitative version of the proof of (\ref{i4}). If $\|x-g_Tx\| \leq \theta\left(\frac\eps2\right)$, then $\|Tx-Tg_Tx\| \leq \frac\eps2$. Therefore, we have that
\begin{align*}
\|x-Tx\| = \|f_Tx-x\| &= \|f_Tx-f_Tg_Tx\| \\
&= \|2(x-g_Tx) - (Tx-Tg_Tx) \| \\
&\leq 2\|x-g_Tx\| + \|Tx-Tg_Tx\| \\
&\leq 2 \cdot \frac\eps4 + \frac\eps2 = \eps.
\end{align*}
\end{enumerate}
\end{proof}

\begin{definition}
If $T : C \to C$ is a continuous pseudocontraction, we define the map $h_T : C \to C$ by putting $h_T:=T$ if $T$ is nonexpansive and $h_T:=g_T$ otherwise.
\end{definition}

The map $h_T$ is defined purely for our convenience, as we could have used $g_T$ regardless of the status of $T$, but we want to emphasize that if $T$ is nonexpansive, then the use of $T$ is sufficient.

\begin{corollary}\label{cor-ht}
Let $T: C \to C$ be a continuous pseudocontraction. Then:
\begin{enumerate}[(i)]
\item $h_T$ is nonexpansive;
\item for all $x \in C$, $\|x-h_Tx\| \leq \|x-Tx\|$;
\item if $T$ is uniformly continuous with modulus $\theta$, then for all $x \in C$ and all $\eps>0$, with $\|x-h_Tx\| \leq \tilde{\theta}(\eps)$, we have that $\|x-Tx\| \leq \eps$;
\item $h_T$ and $T$ have the same fixed points.
\end{enumerate}
\end{corollary}

\section{The proof using limsup's but only $\varepsilon$-infima}\label{section-limsup}

The main focus of this paper is the following theorem (here and below 
$\N^*:=\{ 1,2,3,\ldots\}$). 

\begin{theorem}[{cf. \cite{Rei80}}]\label{main-thm}
Let $X$ be a Banach space which is uniformly convex with modulus $\eta$ and uniformly smooth with modulus $\tau$. Let $C \subseteq X$ a closed, convex, nonempty subset. Let $b \in \N^*$ be such that for all $y \in C$, $\|y\| \leq b$ and the diameter of $C$ is bounded by $b$. Let $T: C \to C$ be a pseudocontraction that is uniformly continuous with modulus $\theta$ and $x \in C$. For all $t \in (0,1)$ put $x_t$ to be the unique point in $C$ such that $x_t = tTx_t + (1-t)x$ (which exists by Propositions~\ref{sps1} and \ref{sps2}). Then for all $(t_n) \subseteq (0,1)$ such that $\lim\limits_{n \to \infty} t_n = 1$ we have that $(x_{t_n})$ is Cauchy.
\end{theorem}

This theorem was first proven by Reich \cite{Rei80} without the assumption of uniform convexity. The starting point of our investigations is the proof given by Morales \cite{Mor90}, which we shall now illustrate, after giving a preliminary lemma.

\begin{lemma}\label{liminf}
Let $(a_n)$, $(b_n)$ be two bounded sequences of reals. Then
$$\liminf_{n \to \infty} (a_n - b_n) \leq \limsup_{n \to \infty} a_n - \limsup_{n \to \infty} b_n.$$
\end{lemma}

\begin{proof}
We have that:
$$\limsup_{n\to \infty} b_n = \limsup_{n \to \infty} (b_n - a_n + a_n) \leq \limsup_{n \to \infty} (b_n-a_n) + \limsup_{n \to \infty} a_n = - \liminf_{n \to \infty} (a_n - b_n) + \limsup_{n \to \infty} a_n.$$
\end{proof}

\ \\ \noindent {\it Proof of the theorem.} We first show that for all $(t_n) \subseteq (0,1)$ such that $\lim\limits_{n \to \infty} t_n = 1$, there exist a $p \in Fix(T)$ and $(n_k)$, strictly increasing, such that $(x_{t_{n_k}}) \to p$. Put, for all $n$, $x_n:=x_{t_n}$.  Then, for all $n$,
$$\|x_n - Tx_n\| = \|t_nTx_n + (1-t_n)x - Tx_n\| = \|(1-t_n)(x-Tx_n)\| \leq (1-t_n)b,$$
so $\lim_{n \to \infty} \|x_n -Tx_n \| = 0$ and therefore (by Corollary~\ref{cor-ht}.(ii)) $\lim_{n \to \infty} \|x_n -h_Tx_n \| = 0$. Define now $F: C \to \R^+$, for all $z \in C$, by $F(z):=\limsup_{n \to \infty} \|x_n-z\|$. Let $K$ be the set of minimizers of $F$.\\[1mm]

\noindent {\bf Claim.} There is a $p \in K \cap Fix(T)$.\\
{\bf Proof of claim:} Since $F$ is convex and continuous, $C$ is closed convex bounded nonempty, and $X$ is uniformly smooth, hence reflexive, we have that $K \neq \emptyset$.  Let $y \in K$ and $z \in C$. Then:
\begin{align*}
F(h_Ty) = \limsup_{n \to \infty} \|x_n - h_Ty\| &\leq \limsup_{n \to \infty}(\|x_n - h_Tx_n\| + \|h_Tx_n - h_Ty\|) \\
&\leq \limsup_{n \to \infty}(\|x_n - h_Tx_n\| + \|x_n - y\|) \\
&\leq \limsup_{n \to \infty}\|x_n - h_Tx_n\| + \limsup_{n \to \infty}\|x_n - y\| \\
&= F(y) \leq F(z),
\end{align*}
so $h_Ty \in K$. Since $K$ is a closed convex bounded nonempty subset of a uniformly smooth space, and it is invariant under the action of the nonexpansive mapping $h_T$, we have that there is a $p \in K \cap Fix(h_T)$, so by Corollary~\ref{cor-ht}.(iv), $p \in K \cap Fix(T)$.
\hfill $\blacksquare$\\[2mm]

We only sketch the remainder of the proof, since the details that will actually be used shall be given later. Let $\eps > 0$ and put $r:= x-p$. Using the continuity of $j$, let $\mu\in(0,1)$ be small enough such that for any $n$, $\langle r, j(x_n-p)\rangle \leq \eps + \langle r,j(x_n-p-\mu r)\rangle$. (Note that $p+\mu r =(1-\mu)p+ \mu x \in C$.) By Lemma~\ref{jl}, we have that $\|x_n-p-\mu r\|^2 \leq \|x_n-p\|^2 + 2\langle -\mu r, j(x_n - p - \mu r)\rangle$. Summing up, we get that
$$\langle r, j(x_n-p) \rangle \leq \eps + \frac1{2\mu} (\|x_n -p \|^2 - \|x_n - p - \mu r\|^2),$$
so by Lemma~\ref{liminf}, $\liminf_{n\to\infty} \langle r, j(x_n-p) \rangle \leq \eps$. Since $\eps$ was chosen arbitrarily and $r=x-p$, we easily get that there is an $(n_k)$, strictly increasing, such that $\limsup_{k \to \infty} \langle x - p, j(x_{n_k} - p) \rangle \leq 0$. We use that $T$ is a pseudocontraction to derive that for any $n$,
\begin{align*}
\langle x_n-x, j(x_n - p) \rangle &= t_n \langle Tx_n - Tp, j(x_n-p) \rangle + t_n \langle p-x, j(x_n -p )\rangle \\
&\leq t_n \|x_n-p\|^2 + t_n \langle p-x, j(x_n - p) \rangle \\
&=t_n\langle x_n-x, j(x_n - p) \rangle,
\end{align*}
so that $\langle x_n - x, j(x_n -p)\rangle \leq 0$, which we sum up with the previous inequality to get that $(x_{t_{n_k}}) \to p$.

To obtain the convergence of $(x_{t_n})$ for any (suitable, from now on) sequence $(t_n)$, it is clear that we need only to show that there is a $p$ such that for any $(t_n)$ there is an $(n_k)$ such that $(x_{t_{n_k}}) \to p$. Fix a canonical sequence, say $s_m:= 1-\frac1{m+1}$, for any $m$. By the previous argument, we have that there is an $(m_l)$ and a $p \in Fix(T)$ such that $(x_{s_{m_l}}) \to p$. Now consider a sequence $(t_n)$. Again, by the above, there is an $(n_k)$ and a $q \in Fix(T)$ such that $(x_{t_{n_k}}) \to q$. What remains to be shown is that $p=q$. Let $\eps>0$. Put, for all $k$, $x_k = x_{t_{n_k}}$. Let $k_0$ be big enough such that $\|x_{k_0} - q\|  \leq \eps^2/4b$ and that (again, using the continuity of $j$) $\langle x_{k_0} - x , j(q-p)-j(x_{k_0} - p) \rangle \leq \eps^2/4$. Using arguments like before, we get that $\langle q-x, j(q-p) \rangle \leq \eps^2/2$ and similarly  $\langle p-x, j(p-q) \rangle \leq \eps^2/2$, so $\|p-q\|\leq\eps$.
\hfill $\Box$\\[2mm]

It is now clear that the least elementary principles are used in the Claim, where appeal is made to results of Banach space theory as established in a set-theoretic framework. From the point of a quantitative analysis the most difficult 
argument is the proof of $K\not=\emptyset$ using the reflexivity of $X.$ This 
can be avoided as follows.
In the light of the conclusion of the theorem, it is immediate that the function $F$ has a unique minimizer, namely the limit of $(x_{t_n})$, so a viable idea would be to get to this uniqueness in an {\it a priori} way. This is where the additional hypothesis of uniform convexity helps us, through Proposition~\ref{l3-a}, which gives a modulus of uniform convexity for the squared norm -- and thus also for $F$ -- that acts as a modulus of minimizer uniqueness for $F$ which 
allows one to show the existence of an actual minimizer as the limit of 
approximate minimizers, with this modulus providing a rate of convergence. 
Let us see how these concepts come into play.

\ \\ \noindent {\it Second proof of the Claim.} 
We divide this proof into a series of claims.\\[1mm]

\noindent {\bf Claim 1.} For all $\eps > 0$ there is a $y \in C$ such that for all $z \in C$:
\begin{itemize}
\item $\limsup\limits_{n \to \infty} \|x_n-y\|^2 \leq \limsup\limits_{n \to \infty} \|x_n-z\|^2 + \eps$;
\item $\limsup\limits_{n \to \infty} \|x_n-h_Ty\|^2 \leq \limsup\limits_{n \to \infty} \|x_n-z\|^2 + \eps$.
\end{itemize}
{\bf Proof of claim:} As before, we have that $\lim_{n \to \infty} \|x_n -h_Tx_n \| = 0$. Suppose that for all $y \in C$ there is a $z \in C$ such that
$$\limsup\limits_{n \to \infty} \|x_n-y\|^2 > \limsup\limits_{n \to \infty} \|x_n-z\|^2 + \eps.$$
Let $\hat{y} \in C$ and put $K:= \left\lceil \frac{b}{\eps} \right\rceil$. Put, then, $f_1:=\hat{y}$ and recursively for all $i \in \{1,\ldots,K\}$ put $f_{i+1}$ such that
$$\limsup\limits_{n \to \infty} \|x_n-f_i\|^2 > \limsup\limits_{n \to \infty} \|x_n-f_{i+1}\|^2 + \eps.$$
Therefore,
$$ b \geq \limsup\limits_{n \to \infty} \|x_n-f_1\|^2 > \limsup\limits_{n \to \infty} \|x_n-f_{K+1}\|^2 + K\eps \geq K\eps \geq b,$$
which is a contradiction. Thus, there is a $y \in C$ such that for all $z \in C$
$$\limsup\limits_{n \to \infty} \|x_n-y\|^2 \leq \limsup\limits_{n \to \infty} \|x_n-z\|^2 + \eps.$$
Now, we have, again as before, that
$$\limsup\limits_{n \to \infty} \|x_n-h_Ty\| \leq  \limsup\limits_{n \to \infty}\|x_n-y\|,$$
so, for all $z \in C$,
$$\limsup\limits_{n \to \infty} \|x_n-h_Ty\| ^2 \leq \limsup\limits_{n \to \infty} \|x_n-y\| ^2 \leq \limsup\limits_{n \to \infty} \|x_n-z\|^2 + \eps.$$
\hfill $\blacksquare$\\[2mm]

\noindent {\bf Claim 2.} For all $\eps > 0$ there is a $u \in C$ such that:
\begin{itemize}
\item for all $z \in C$, $\limsup\limits_{n \to \infty} \|x_n-u\|^2 \leq \limsup\limits_{n \to \infty} \|x_n-z\|^2 + \eps$;
\item $\|u-Tu\| \leq \eps$.
\end{itemize}
{\bf Proof of claim:} Take $\eta_1:= \min\left(\eps, \frac12 \psi_{b,\eta}(\tilde{\theta}(\eps)) \right) > 0$. Apply Claim 1 for $(t_n)$ and $\eta_1$ and put $u$ to be the resulting $y$. We have only to show that $\|u-h_Tu\|\leq \tilde{\theta}(\eps)$, since from that, using Corollary~\ref{cor-ht}.(iii), it follows that $\|u-Tu\|\leq\eps$. Suppose not. Then, for all $n$,
$$\|(x_n-u) - (x_n-h_Tu)\| = \|u-h_Tu\| \geq \tilde{\theta}(\eps),$$
so, for all $n$, by Proposition~\ref{l3-a},
$$\left\|x_n - \left( \frac{u+h_Tu}2 \right) \right\|^2 + \psi_{b,\eta}(\tilde{\theta}(\eps)) \leq \frac12 \|x_n-u\|^2  + \frac12 \|x_n-h_Tu\|^2.$$
Then, applying the defining property of $u$, we get that
\begin{align*}
\limsup\limits_{n \to \infty} \left\|x_n-\left( \frac{u+h_Tu}2 \right) \right\|^2 + \psi_{b,\eta}(\tilde{\theta}(\eps)) &\leq \frac12 \limsup\limits_{n \to \infty} \|x_n-u\|^2 + \frac12\limsup\limits_{n \to \infty} \|x_n-h_Tu\|^2 \\
&\leq \limsup\limits_{n \to \infty} \left\|x_n-\left( \frac{u+h_Tu}2 \right) \right\|^2  + \eta_1,
\end{align*}
which is a contradiction, since $0 < \eta_1 \leq \frac12 \psi_{b,\eta}(\tilde{\theta}(\eps)) $.
\hfill $\blacksquare$\\[2mm]

Again, we only sketch the remainder of this proof. For any $m\in\N^*$, we fix a $u_m$ such that for all $z \in C$, $\limsup_{n \to \infty} \|x_n-u_m\|^2 \leq \limsup_{n \to \infty} \|x_n-z\|^2 + 1/m$ and $\|u_m-Tu_m\| \leq 1/m$. We show that $(u_m)$ is Cauchy. Let $\eps > 0$ and let $m$, $p \geq \lceil 2/\psi_{b,\eta}(\eps) \rceil$. Assume that $\|u_m - u_p\| > \eps$. Then, for all $n$, using Proposition~\ref{l3-a} as before,
\begin{align*}
\limsup\limits_{n \to \infty} \left\|x_n-\left( \frac{u_m+u_p}2 \right) \right\|^2 + \psi_{b,\eta}(\eps) &\leq \frac12 \limsup\limits_{n \to \infty} \|x_n-u_m\|^2 + \frac12\limsup\limits_{n \to \infty} \|x_n-u_p\|^2 \\
&\leq \limsup\limits_{n \to \infty} \left\|x_n-\left( \frac{u_m+u_p}2 \right) \right\|^2  + \frac12 \psi_{b,\eta}(\eps),
\end{align*}
which is a contradiction. It is then immediate that the limit of $(u_m)$ satisfies our requirements.
\hfill $\Box$\\[2mm]

The next principles we can now remove from the proof are the ones that allowed us, for example, to pass to the limit in the argument above. What we do is to show that the approximate solutions obtained in Claim 2 are enough for the whole line of argument to go through, by essentially removing any ideal point that would appear in the course of the proof by an approximate one. This is made possible again by the use of Proposition~\ref{l3-a}, asserting  that two $\delta$-infima of $F$, for sufficiently small $\delta$, must be $\varepsilon$-close. Also, it is now clear that the resulting proof does no longer use the existence of $F$ as a function but only the existence of the individual limsup's in the form 
$$\forall z\in C\,\exists a\in\R^+ \ (a=\limsup_{n\to\infty} \| x_{t_n}-z\|^2).$$
As a result of this, the proof may be formalized in a deductive system to which the logical bound extraction theorems, mentioned in the Introduction, apply -- which is not clear if $F$ would be needed as an object (see also Remark~\ref{rem-logic1} below).

\ \\ \noindent {\it Proof of the theorem using only the aforementioned principles.}
We presuppose the truth of Claim 2 in the previous proof, i.e. that for all $(t_n) \subseteq (0,1)$ such that $\lim\limits_{n \to \infty} t_n = 1$ and for all $\eps > 0$ there is a $u \in C$ such that:
\begin{itemize}
\item for all $z \in C$, $\limsup\limits_{n \to \infty} \|x_{t_n}-u\|^2 \leq \limsup\limits_{n \to \infty} \|x_{t_n}-z\|^2 + \eps$;
\item $\|u-Tu\| \leq \eps$.
\end{itemize}
Thus, we shall start the numbering of claims at 3.\\[1mm]

\noindent {\bf Claim 3.} For all $(t_n) \subseteq (0,1)$ such that $\lim\limits_{n \to \infty} t_n = 1$ and for all $\eps > 0$ there is a $v \in C$ such that:
\begin{itemize}
\item for all $z \in C$, $\limsup\limits_{n \to \infty} \|x_{t_n}-v\|^2 \leq \limsup\limits_{n \to \infty} \|x_{t_n}-z\|^2 + \eps$;
\item for all $t\in(0,1)$, $\langle x_t-x, j(x_t-v) \rangle \leq \eps$.
\end{itemize}
{\bf Proof of claim:} Take $\eta_2:=\min\left(\eps, \frac12\psi_{b,\eta}\left(\omega_\tau\left(b,\frac{\eps}{2b}\right)\right)\right).$

Apply Claim 2 for $(t_n)$ and $\eta_2$ and put $v$ to be the resulting $u$.

We have to show that for all $t \in (0,1)$, $\langle x_t-x, j(x_t-v) \rangle \leq \eps$. Let $t \in (0,1)$ and put $\delta:=\min\left(\eta_2,\frac{\eps(1-t)}{2b}\right)$. Apply Claim 2 for $(t_n)$ and $\delta$ and put $v'$ to be the resulting $u$, so in particular $\|v'-Tv'\|\leq\delta$. We then obtain that
\begin{align*}
\langle x_t-x, j(x_t-v') \rangle &= \langle t(Tx_t-x), j(x_t-v') \rangle \\
&=t\langle Tx_t -Tv', j(x_t-v') \rangle + t \langle Tv'-v', j(x_t -v') \rangle + t\langle v'-x, j(x_t - v') \rangle \\
&\leq t\|x_t-v'\|^2 + t \|Tv'-v'\| \|x_t-v'\| + t\langle v'-x, j(x_t - v') \rangle \\
&\leq t\|x_t-v'\|^2  + t\langle v'-x, j(x_t - v') \rangle + t\delta b \\
&\leq t\langle x_t-v', j(x_t-v') \rangle + t \langle v'-x,j(x_t -v') \rangle + \delta b \\
&= t \langle x_t - x, j(x_t - v') \rangle + \delta b,
\end{align*}
from which we get that
\begin{equation}\label{e3.1}
\langle x_t-x, j(x_t-v') \rangle \leq \frac{\delta b}{1-t} \leq \frac{\eps}2.
\end{equation}
Suppose that $\|v-v'\| \geq \omega_\tau(b,\frac{\eps}{2b})$, so, for all $n$,
$$\left\|x_n - \frac{v+v'}2\right\|^2 + \psi_{b,\eta}\left(\omega_\tau\left(b,\frac{\eps}{2b}\right)\right) \leq \frac12\|x_n-v\|^2 + \frac12\|x_n-v'\|^2.$$
Then
\begin{align*}
\limsup\limits_{n \to \infty} \left\|x_n-\left( \frac{v+v'}2 \right) \right\|^2 + \psi_{b,\eta}\left(\omega_\tau\left(b,\frac{\eps}{2b}\right)\right)  &\leq \frac12 \limsup\limits_{n \to \infty} \|x_n- v\|^2 + \frac12\limsup\limits_{n \to \infty} \|x_n-v'\|^2 \\
&\leq \limsup\limits_{n \to \infty} \left\|x_n-\left( \frac{v+v'}2 \right) \right\|^2  + \eta_2,
\end{align*}
which is a contradiction. So $\|v-v'\| \leq \omega_\tau(b,\frac{\eps}{2b})$, i.e. $\|(x_t-v)-(x_t-v')\| \leq \omega_\tau(b,\frac{\eps}{2b})$. From that we obtain
$$\|j(x_t-v)-j(x_t-v')\| \leq \frac{\eps}{2b}$$
and
\begin{equation}\label{e3.2}
\langle x_t-x, j(x_t-v)-j(x_t-v') \rangle \leq \frac{\eps}{2b} \cdot b = \frac\eps2.
\end{equation}
From \eqref{e3.1} and \eqref{e3.2} we derive our conclusion.
\hfill $\blacksquare$\\[2mm]

\noindent {\bf Claim 4.} For all $(t_n) \subseteq (0,1)$ such that $\lim\limits_{n \to \infty} t_n = 1$ and for all $\eps > 0$ there is a $w \in C$ such that:
\begin{itemize}
\item for all $t\in(0,1)$, $\langle x_t-x, j(x_t-w) \rangle \leq \eps$;
\item there exists $(n_k)$, strictly increasing, such that $\limsup\limits_{k \to \infty} \|x_{t_{n_k}} - w\|^2 \leq \eps$.
\end{itemize}
{\bf Proof of claim:} Put $\mu:=\min\left(\frac{\omega_\tau\left(b,\frac\eps{3b}\right)}b, \frac14\right)$ and $\eta_3:=\min\left(\frac\eps3, 2\mu\cdot\frac\eps3\right).$

Apply Claim 3 for $(t_n)$ and $\eta_3$ and put $w$ to be the resulting $v$. Denote, for all $n$, $x_n:=x_{t_n}$. We have that, for all $n$,
\begin{equation}\label{e4.3}
\langle x_n-x, j(x_n-w) \rangle \leq \frac\eps3.
\end{equation}
Put $q:=x-w$. Since $\mu \in (0,1)$, $w+\mu q = (1-\mu)w+\mu x \in C$. By Lemma~\ref{jl}, we have that
\begin{equation}\label{e4.1}
\|x_n -w-\mu q\|^2 \leq \|x_n-w\|^2 + 2\langle -\mu q, j(x_n-w-\mu q) \rangle.
\end{equation}
Since
$$\|(x_n-w)-(x_n-w-\mu q)\| = \|\mu q\| \leq \mu b \leq \omega_\tau\left(b,\frac\eps{3b} \right),$$
we have that
$$\|j(x_n-w)-j(x_n-w-\mu q)\| \leq \frac\eps{3b},$$
and so that
\begin{equation}\label{e4.2}
\langle q, j(x_n-w) \rangle \leq \frac\eps3 + \langle q, j(x_n-w-\mu q) \rangle.
\end{equation}
From \eqref{e4.1} and \eqref{e4.2} we get that
$$\langle q, j(x_n-w) \rangle \leq \frac\eps 3 + \frac1{2\mu} (\|x_n-w\|^2 - \|x_n-w-\mu q\|^2).$$
Applying Lemma~\ref{liminf}, we obtain that
\begin{align*}
\liminf_{n \to \infty} \langle q, j(x_n-w) \rangle &\leq \frac\eps3 + \frac1{2\mu}\left(\limsup_{n \to \infty} \|x_n-w\|^2 - \limsup_{n\to\infty}\|x_n-w-\mu q\|^2\right) \\
&\leq \frac\eps3 + \frac1{2\mu} \cdot 2\mu \cdot \frac\eps3 = \frac{2\eps}3,
\end{align*}
and therefore that there exists $(n_k)$, strictly increasing, such that 
$$\lim\limits_{k \to \infty} \langle q, j(x_{n_k}-w) \rangle \leq \frac{2\eps}3,$$
so in particular, noting also that $q=x-w$,
$$\limsup\limits_{k \to \infty} \langle x-w, j(x_{n_k}-w) \rangle \leq \frac{2\eps}3. $$
Using \eqref{e4.3}, we derive $\limsup\limits_{k \to \infty} \|x_{n_k} - w\|^2 \leq \eps$, i.e. our conclusion.
\hfill $\blacksquare$\\[2mm]

\noindent {\bf Claim 5.} For all $\eps>0$ there is a $g \in C$ such that for all $(t_n) \subseteq (0,1)$ with $\lim\limits_{n \to \infty} t_n = 1$, there exists $(n_k)$, strictly increasing, such that
$$\limsup\limits_{k \to \infty} \|x_{t_{n_k}} - g\|\leq \eps.$$
{\bf Proof of claim:} Put
$$\eta_4:=\min \left\{ \frac{\eps^2}{24}, \frac{\eps^4}{48^2b^2}, \frac{\omega_\tau^2\left(b,\frac{\eps^2}{24b}\right)}4 \right\}$$
and, for all $m$, $s_m:=1-\frac1{m+1}$.

Apply Claim 4 for $(s_m)$ and $\eta_4$ and put $g$ to be the resulting $w$. In particular, there is $(m_l)$, strictly increasing, such that $\limsup\limits_{l \to \infty} \|x_{s_{m_l}} - g\|^2 \leq \eta_4$.

Let now $(t_n)$ be chosen arbitrarily such that $\lim\limits_{n \to \infty} t_n = 1$. Apply Claim 4 for $(t_n)$ and $\eta_4$ and put $g'$ to be the resulting $w$. In particular, there is $(n_k)$, strictly increasing, such that $\limsup\limits_{k \to \infty} \|x_{t_{n_k}} - g'\|^2 \leq \eta_4$.

We have that for all $k$,
$$\langle x_{t_{n_k}}-x, j(x_{t_{n_k}} -g ) \rangle \leq \eta_4 \leq \frac{\eps^2}{24}.$$

Take a $k_0$ sufficiently large such that
$$\|x_{t_{n_{k_0}}} - g'\| \leq \limsup_{k \to \infty} \|x_{t_{n_k}} - g'\| + \min\left(\frac{\eps^2}{48b},\frac{\omega_\tau\left(b,\frac{\eps^2}{24b}\right)}2\right).$$
We have that
$$\|x_{t_{n_{k_0}}} - g'\| \leq \frac{\eps^2}{48b} + \frac{\eps^2}{48b} = \frac{\eps^2}{24b}$$
and that
$$\|(x_{t_{n_{k_0}}} - g) - (g'-g) \| = \|x_{t_{n_{k_0}}} - g'\| \leq \frac{\omega_\tau\left(b,\frac{\eps^2}{24b}\right)}2 + \frac{\omega_\tau\left(b,\frac{\eps^2}{24b}\right)}2 = \omega_\tau\left(b,\frac{\eps^2}{24b}\right),$$
so
$$\|j(x_{t_{n_{k_0}}} - g) - j(g'-g) \| \leq \frac{\eps^2}{24b}.$$
Therefore
\begin{align*}
\langle g'-x, j(g'-g) \rangle &\leq \langle g' - x_{t_{n_{k_0}}}, j(g'-g) \rangle + \langle x_{t_{n_{k_0}}} - x, j(g'-g) \rangle \\
&\leq b \cdot \|x_{t_{n_{k_0}}}-g'\| + \langle x_{t_{n_{k_0}}} - x, j(x_{t_{n_{k_0}}} - g) \rangle \\
&\quad+ \langle x_{t_{n_{k_0}}} - x, j(g'-g) - j(x_{t_{n_{k_0}}}-g) \rangle \\
&\leq b \cdot \frac{\eps^2}{24b} + \frac{\eps^2}{24} + b\cdot \frac{\eps^2}{24b} = \frac{\eps^2}8.
\end{align*}
Similarly, we obtain that
$$\langle g-x, j(g-g') \rangle \leq \frac{\eps^2}8.$$
Summing up, we get that $\langle g-g', j(g-g')\rangle \leq \frac{\eps^2}4$, so $\|g-g'\| \leq \frac\eps2$.
Since $\limsup_{k \to \infty} \|x_{t_{n_k}} - g'\|^2 \leq \frac{\eps^2}{24} \leq \frac{\eps^2}4$, we have
$$\limsup_{k \to \infty}\|x_{t_{n_k}} -g\| \leq \limsup_{k \to \infty}\|x_{t_{n_k}} -g'\| + \|g - g'\| \leq \frac\eps2+ \frac\eps2 = \eps,$$
i.e. our conclusion.\hfill $\blacksquare$\\[2mm]

\noindent {\bf Claim 6.} For all $\eps>0$ there is an $h \in C$ such that for all $(t_n) \subseteq (0,1)$ with $\lim\limits_{n \to \infty} t_n = 1$, we have that
$$\limsup\limits_{n \to \infty} \|x_{t_n} - h\|\leq \eps.$$
{\bf Proof of claim:} Apply Claim 5 for $\eps$ and put $h$ to be the resulting $g$. Let now $(t_n)$ be chosen arbitrarily such that $\lim\limits_{n \to \infty} t_n = 1$.

Suppose that $\limsup\limits_{n \to \infty} \|x_{t_n} - h\|>\eps$. Then there is an $\eta > 0$ such that for all $N$ there is an $n \geq N+1$ such that $\|x_{t_n} - h\| > \eps + \eta$, so there is an $(n_k)$, strictly increasing, such that for all $k$, $\|x_{t_{n_k}} - h\| > \eps + \eta$. By the defining property of $h$, we get that there is a $(k_l)$, strictly increasing, such that
$$\limsup\limits_{n \to \infty} \|x_{t_{n_{k_l}}} - h\|\leq \eps,$$
so there is an $L$ such that for all $l \geq L$,
$$\|x_{t_{n_{k_l}}} - h\| \leq \eps + \eta,$$
which contradicts the defining property of $(n_k)$.\hfill $\blacksquare$\\[2mm]

\noindent {\bf Claim 7.} For all $(t_n) \subseteq (0,1)$ with $\lim\limits_{n \to \infty} t_n = 1$, we have that the sequence $(x_{t_n})$ is Cauchy.\\[1mm]
{\bf Proof of claim:} Denote, for all $n$, $x_n:=x_{t_n}$. We want to show that for all $\eps >0$ there is an $N$ such that for all $m,n \geq N$, $\|x_n - x_m \| \leq \eps$. Let $\eps > 0$. By applying Claim 6 for $\frac\eps4$, we obtain an $h \in C$ having the property that there is an $N$ such that for all $n \geq N$,
$$\|x_n - h\| \leq \frac\eps4 + \frac\eps4 = \frac\eps2.$$
Take $n, m \geq N$. Then
$$\|x_n - x_m \| \leq \|x_n - h\| + \|x_m -h\| \leq \frac\eps2 + \frac\eps2 = \eps,$$
i.e. our conclusion.\hfill $\blacksquare$\\[2mm]

This last claim is exactly our desired statement.
\hfill $\Box$

\begin{remark}[for logicians; we use the terminology from \cite{Koh08}]\label{rem-logic1} \rm
An inspection of the proof of the Cauchy property and hence of the 
metastability of $(x_{t_n})$ in this 
section  shows that it can be carried out in the formal system WE-PA$^{\omega}[X,\|\cdot\|,\eta,J_X,\omega_X,C]+$CA$_{ar}$ 
where   WE-PA$^{\omega}[X,\|\cdot\|,\eta,J_X,\omega_X,C]$ is
 defined as in \cite[(17.68)]{Koh08} and then augmented by the normalized duality mapping $J_X$ and the modulus of 
uniform smoothness $\omega_X$ as 
in \cite{KohLeu12}. CA$_{ar}$ denotes the schema 
of arithmetic comprehension which is needed to show the existence of 
$\limsup\limits_{n\to\infty}  \| x_{t_n}-y\|^2.$ From the 
logical metatheorems in \cite{Koh08,KohLeu12} and Theorems 11.11 and 11.13 in 
\cite{Koh08} it follows that one can extract a rate of metastability for 
the Cauchy property of $(x_{t_n})$ which is definable in G\"odel's calculus 
$T$ of primitive recursive functionals of finite type augmented by Spector's 
bar recursion $B_{0,1}$ of lowest type. In the next section we will show that 
even the use of $B_{0,1}$ can be avoided.
\end{remark}

\section{The proof using approximate limsup's}\label{section-approx}

In this section we, in particular, show that the use of limsup's can be replaced by that of 
$\varepsilon$-limsup's whose existence can be established by induction 
(for logicians: $\Pi^0_2$-induction). As a result of this, the proof can 
even be formalized without arithmetic comprehension and so the extractability 
of a primitive recursive (in the sense of G\"odel's $T$) rate of metastability 
is guaranteed (see Remark~\ref{rem-logic1}). We also exhibit the finitary 
content of the actual use of approximate limsup's made in the proof.
\begin{remark}\label{rem-ar} \rm
The process of eliminating $\limsup$'s by an arithmetical principle in this 
section is in line with {\rm \cite[Proposition 
5.9]{Koh97A}} where such an arithmetization of the use of limsup's to $T_1$ is shown to 
be possible in a certain restrictive deductive context and by  
{\rm \cite[Theorem 6.1]{Koh00}} it is optimal. In {\rm \cite[Section 17.9]{Koh08}}, the 
approach is shown to work also within the framework of abstract spaces.
In our context, however, we cannot directly apply these results as 
the limsup's are used in the presence of e.g. inductions going beyond 
quantifier-free induction. However, as usual in the context of 
ordinary proofs, the arithmetization can nevertheless  
be carried out without problems and we suspect that this could be 
explained in terms of logical metatheorems by treating the inductions 
used as implicative assumptions and using that the method behind 
these arithmetizations works for arbitrary (arithmetical) formulas as 
long as certain monotonicity conditions are fulfilled (see {\rm \cite{Koh98}}).
Nevertheless, we leave this for future research to clarify. 
\end{remark}
\subsection{The arithmetized version of limits superior}

\begin{definition}
Let $(a_n)$ be a sequence of reals and $\eps > 0$. A number $a \in \R$ is called an {\bf $\eps$-approximate limsup} (or simply an {\bf $\eps$-limsup}) for $(a_n)$ if:
\begin{itemize}
\item for all $n$ there is an $m$ such that $a_{n+m} \geq a-\eps$;
\item there is a $j$ such that for all $l$, $a_{j+l} \leq a+\eps$.
\end{itemize}
\end{definition}

What makes approximate limsup's suitable for proof mining is that they admit an existence proof which uses only $\Pi^0_2$-induction.

\begin{proposition}[$\Pi^0_2$-IA]\label{exlimsup}
For all $b, k \in \N$ and for all sequences of reals $(a_n)$ contained in the interval $[0,b]$, there is a $p \in \N$ with $0 \leq p \leq b\cdot (k+1)$ such that $\frac{p}{k+1}$ is a $\frac1{k+1}$-limsup of $(a_n)$.
\end{proposition}

\begin{proof}
Let $b$, $k$ and $(a_n)$ be as in the statement.\\[1mm]

\noindent {\bf Claim.} There is a $p \in \N$ with $0 \leq p \leq b\cdot (k+1)$ such that it is not the case that for all $j$ there is an $l$ with $a_{j+l} > \frac{p-1}{k+1}$ implies that for all $j$ there is an $l$ with $a_{j+l} > \frac{p}{k+1}$.\\[1mm]
{\bf Proof of claim:} Assume towards a contradiction that the opposite holds, i.e. for all natural numbers $p$ smaller or equal to $b \cdot (k+1)$, we have that $Q(p)$ implies $Q(p+1)$, where $Q(p)$ is the $\Pi^0_2$ statement that for all $j$ there is an $l$ such that $a_{j+l} > \frac{p-1}{k+1}$. Since $Q(0)$ holds trivially, we have by $\Pi^0_2$-induction that $Q(b \cdot (k+1) +1)$. But that states that for all $j$ there is an $l$ such that $a_{j+l} >b$, clearly false.
\hfill $\blacksquare$\\[0.5mm]

Take $p$ as in the Claim. Then $0 \leq p \leq b\cdot (k+1)$ and:
\begin{enumerate}[(i)]
\item for all $n$ there is an $m$ such that $a_{n+m} > \frac{p-1}{k+1}$, so $ a_{n+m}\geq  \frac{p}{k+1} - \frac1{k+1} $,
\item there is a $j$ such that for all $l$, $a_{j+l} \leq  \frac{p}{k+1}$, so $a_{j+l} \leq  \frac{p}{k+1}+\frac1{k+1}$,
\end{enumerate}
i.e.  $ \frac{p}{k+1}$ is a $\frac1{k+1}$-limsup of $(a_n)$.
\end{proof}

\begin{remark} \rm
One can even show, as mentioned in the Introduction, that this statement is equivalent to $\Pi^0_2$-induction. To do that, we tweak the argument used in the proof of \cite[Theorem 6.1]{Koh00}, whose statement affirms that the existence of limsup's (without function parameters) implies $\Pi^0_2$-induction, to also work with rational approximate limsup's, i.e. in the form given in Proposition~\ref{exlimsup}. The limsup hypothesis is used two times: once in Claims 1-3 and once when it yields $\Sigma^0_1$-induction as the first stage of a bootstrapping process. The second application does not pose any serious problems, while the first one is a bit more involved, since the statements of Claims 1-3 must be adjusted. Set, for any $k \in \N^*$, $L(k):=4k(k+1)>0$. One then requires in Claims 2 and 3 from $a$ to be a rational $\frac1{L(k)}$-limsup and a rational $\frac1{L(k+1)}$-limsup of $(q^f_n)$, respectively, while the new Claim 1 states that for any $k, p\in \N^*$ with $p \leq k$ and for any rational $a \in [0,1]$ which is a $\frac1{L(k)}$-limsup of $(q^f_n)$, the following are equivalent:
\begin{enumerate}[(i)]
\item $a \geq \frac1p - \frac1{L(k)}$;
\item $a > \frac1{p+1} + \frac1{L(k)}$;
\item for all $n$ there is an $m \geq n$ with $f(m) < p$.
\end{enumerate}
The proof then goes through.
\end{remark}

It is not sufficient that one can prove the existence of approximate limsup's, we must also show that they can play the role that is required of them. The following lemma is crucial in this regard, as it proves that one can extract specific sequence ranks that are needed in a later analysis of a proof, whereas the values of the approximate limsup's may be discarded.

\begin{lemma}\label{lpe}
Let $\eps > 0$. Let $(a_n)$, $(b_n)$ and $(c_n)$ be sequences of reals and $q$, $q'$ and $r$ be $\frac\eps4$-limsup's of them, respectively. If $q \leq r+\frac\eps2$ and $q'\leq r+\frac\eps2$, then for all $N$ there is a $k$ such that $a_{N+k} \leq c_{N+k} + \eps$ and $b_{N+k} \leq c_{N+k} + \eps$.
\end{lemma}

\begin{proof}
By the definition of the approximate limsup, we have that:
\begin{itemize}
\item there is a $j$ such that for all $l$, $a_{j+l} \leq q+\frac\eps4$;
\item there is a $j'$ such that for all $l$, $b_{j'+l} \leq q' + \frac\eps4$;
\item for all $n$ there is an $m$ such that $c_{n+m} \geq r-\frac\eps4$, and in the following we denote this $m$ depending on $n$ as $m_n$.
\end{itemize}
Let $N \in \N$. We set $k:=j+j'+m_{N+j+j'}$. Then we have that
$$a_{N+k}=a_{N+j+j'+m_{N+j+j'}} \leq q+\frac\eps4 \leq r + \frac{3\eps}4 \leq c_{N+j+j'+m_{N+j+j'}} + \eps = c_{N+k}+ \eps, $$
and similarly, that $b_{N+k} \leq c_{N+k} + \eps$.
\end{proof}

We will be using the following weaker forms of the above lemma.

\begin{corollary}\label{c1}
Let $\eps > 0$. Let $(a_n)$ and $(c_n)$ be sequences of reals and $q$ and $r$ be $\frac\eps4$-limsup's of them, respectively. If $q \leq r+\frac\eps2$, then for all $N$ there is a $k $ such that $a_{N+k} \leq c_{N+k} + \eps$.
\end{corollary}

\begin{corollary}\label{c2}
Let $\eps > 0$. Let $(a_n)$, $(b_n)$ and $(c_n)$ be sequences of reals and $q$, $q'$ and $r$ be $\frac\eps4$-limsup's of them, respectively. If $q \leq r+\frac\eps2$ and $q'\leq r+\frac\eps2$, then there is a $k $ such that $a_{k} \leq c_{k} + \eps$ and $b_{k} \leq c_{k} + \eps$.
\end{corollary}

\subsection{Replacing limsup's by approximate limsup's}

We consider, in this section, $\alpha : \N \to \N$ and $\gamma : \N \to \N^*$ such that:
\begin{itemize}
\item for all $n$ and all $m \geq \alpha(n)$, $t_m \geq 1 - \frac1{n+1}$;
\item for all $n$, $t_n \leq 1 - \frac1{\gamma(n)}$.
\end{itemize}
In the case that for all $n$, $t_n = 1-\frac1{n+1}$, we may take, for all $n$, $\alpha(n):=n$ and $\gamma(n):=n+1$.

\ \\ \noindent {\it New proof of the theorem.} 
Again, we divide the proof into a series of claims.\\[1mm]

\noindent {\bf Claim I.} Let $(s_n) \subseteq (0,1)$ and $\eps >0$. Then there is a $y \in C$ and a $q \in \Q$ such that $q$ is an $\frac\eps4$-limsup of $(\|x_{s_n}-y\|^2)$ and such that for all $z \in C$ and $r \in \Q$ with $r$ being an $\frac\eps4$-limsup of $(\|x_{s_n}-z\|^2)$, $q \leq r + \frac\eps2$.\\[1mm]
{\bf Proof of claim:} Denote, for all $n$, $x_n:=x_{s_n}$. Assume towards a contradiction that for all $y\in C$ and $q \in \Q$ with $q$ being an $\frac\eps4$-limsup of $(\|x_n-y\|^2)$ there is a $z \in C$ and an $r \in \Q$ such that $r$ is an $\frac\eps4$-limsup of $(\|x_n-z\|^2)$ and $r < q - \frac\eps2$. Let now $z_1 \in C$ be arbitrary and $r_1$ be an $\frac\eps4$-limsup of $(\|x_n-z_1\|^2)$. Put, for all $n\leq\left\lceil\frac{2b^2}{\eps}\right\rceil + 2$, $z_{n+1}$ and $r_{n+1}$ be the $z$ and the $r$ obtained by applying the assumption to $z_n$ and $r_n$ playing the roles of $y$ and $q$, respectively. Then, since $r_1 \leq b^2+\frac\eps2$, and by the assumption, for each $n \leq \left\lceil\frac{2b^2}{\eps}\right\rceil + 2$, we have that $r_{n+1} < r_n - \frac\eps2$, we get that for all $n \leq \left\lceil\frac{2b^2}{\eps}\right\rceil + 3$, $r_n \leq b^2 - (n-2)\frac\eps2$. If we choose $n:= \left\lceil\frac{2b^2}{\eps}\right\rceil + 3$, we obtain that $b^2 - (n-2)\frac\eps2 \leq -\frac\eps2$, contradicting the fact that $r_n$ is an $\frac\eps4$-limsup of a sequence of nonnegative reals.\hfill $\blacksquare$\\[2mm]

Denote, for all $n$, $x_n:=x_{t_n}$. We shall prove the Cauchyness of the sequence in its ``metastable'' formulation, namely: for all $\eps>0$ and all $g: \N \to \N$ there is an $N$ such that $\|x_N - x_{N+g(N)}\| \leq \eps$. Let, therefore, $\eps >0$ and $g: \N \to \N$. From this point on, we shall use the following notations (where $n,c,d \in \N$ and $p \in C$):
\begin{align*}
 s_{p,g}(n)&:=
  \begin{cases} 
       n, \hfill & \text{if $\|x_{n+g(n)} - p\| \leq \|x_n-p\|$}, \\
       n+g(n), \hfill & \text{otherwise.} \\
  \end{cases}\\
  x^p_n & := x_{s_{p,g}(n)}\\
  \nu_4(\eps) &:= \min\left\{ \frac{\eps^4}{9216b^2}, \omega_\tau^2\left(b,\frac{\eps^2}{96b}\right), \frac{\eps^2}{16} \right\} \\
  \delta(\eps)&:= \min \left\{ \frac{\omega_\tau\left(b,\frac{\nu_4(\eps)}{3b}\right)}{b}, \frac14 \right\} \\
  p(\eps) &:= \min \left\{ \frac{\nu_4(\eps)}3, \frac{\eps^2}{96} \right\}\\
  \nu_2(\eps) &:= \frac12\psi_{b,\eta}\left(\omega_\tau\left(b,\frac{p(\eps)}{2b}\right)\right)\\
  \beta(c,\eps) &:= \frac{p(\eps)}{2b\gamma(c)} \\
  q(p,c,d,\eps) &:= \min \left\{\beta(c,\eps), \beta(s_{p,g}(d),\eps) \right\}\\
  \nu_1(p,c,d,\eps) &:= \frac12\psi_{b,\eta}(\tilde{\theta}(q(p,c,d,\eps))).
\end{align*}

\noindent {\bf Claim II.} There are $w$, $w'$, $v$, $v'\in C$ and $k$, $k'$, $l$, $l'$, $h$, $h' \in \N$ such that:
\begin{itemize}
\item $\|x_k -w\|^2 - \|x_k-w-\delta(\eps)(x-w)\|^2$,\\ $\|x^w_{k'}-w'\|^2 - \|x^w_{k'}-w'-\delta(\eps)(x-w')\|^2 \leq \frac{2\nu_4(\eps)\delta(\eps)}3$;
\item $\|x_l-w\|^2 - \left\|x_l -\frac{v+w}2\right\|^2$, $\|x_l-v\|^2 - \left\|x_l -\frac{v+w}2\right\|^2$, \\$\|x_{l'}-w'\|^2 - \left\|x_{l'} -\frac{v'+w'}2\right\|^2$, $\|x_{l'}-v'\|^2 - \left\|x_{l'} -\frac{v'+w'}2\right\|^2 \leq \nu_2(\eps)$;
\item $h,h' \geq \alpha\left(\left\lceil\max\left\{\frac{2b}{\sqrt{\nu_1(w,k,k',\eps)}},\frac{8b^2}{\nu_1(w,k,k',\eps)}\right\}\right\rceil\right)$;
\item $\|x_{h}-v\|^2 - \left\|x_{h} -\frac{v+h_Tv}2\right\|^2$, $\|x_{h'}-v'\|^2 - \left\|x_{h'} -\frac{v'+h_Tv'}2\right\|^2 \leq \frac{\nu_1(w,k,k',\eps)}2$.
\end{itemize}
{\bf Proof of claim:}
\begin{enumerate}[A.]
\item The construction of $w$ and $k$.

We apply Claim I for $(t_n)$ and $u:=\min \left\{ \frac{2\nu_4(\eps)\delta(\eps)}3, \nu_2(\eps) \right\}$. We obtain $w \in C$ and $q_w \in \Q$ such that $q_w$ is an $\frac{u}4$-limsup of $(\|x_n-w\|^2)$ and for all $z \in C$ and $q_z \in \Q$ with $q_z$ being an $\frac{u}4$-limsup of $(\|x_n-z\|^2)$ we have that $q_w \leq q_z + \frac{u}2$.

By the above applied to $z:=w+\delta(\eps)(x-w)$ and $q_z$ an $\frac{u}4$-limsup of $(\|x_n-z\|^2)$, we get after using Corollary~\ref{c1} that there is a $k$ such that
$$\|x_k-w\|^2 - \|x_k-w-\delta(\eps)(x-w)\|^2 \leq u \leq \frac{2\nu_4(\eps)\delta(\eps)}3.$$

\item The construction of $w'$ and $k'$.

We apply Claim I for $(t_{s_{w,g}(n)})$ and $u$. We obtain $w' \in C$ and $q_{w'} \in \Q$ such that $q_{w'}$ is an $\frac{u}4$-limsup of $(\|x^w_n-w'\|^2)$ and for all $z \in C$ and $q_z \in \Q$ with $q_z$ being an $\frac{u}4$-limsup of $(\|x^w_n-z\|^2)$ we have that $q_{w'} \leq q_z + \frac{u}2$.

By the above applied to $z:=w'+\delta(\eps)(x-w')$ and $q_z$ an $\frac{u}4$-limsup of $(\|x^w_n-z\|^2)$, we get after using Corollary~\ref{c1} that there is a $k'$ such that
$$\|x^w_{k'}-w'\|^2 - \|x^w_{k'}-w'-\delta(\eps)(x-w')\|^2 \leq u \leq \frac{2\nu_4(\eps)\delta(\eps)}3.$$

\item The construction of $v$ and $h$.

We apply Claim I for $(t_n)$ and $u':=\min \left\{ \frac{\nu_1(w,k,k',\eps)}2, \nu_2(\eps) \right\}$. We obtain $v \in C$ and $q_v \in \Q$ such that $q_v$ is an $\frac{u'}4$-limsup of $(\|x_n-v\|^2)$ and for all $z \in C$ and $q_z \in \Q$ with $q_z$ being an $\frac{u'}4$-limsup of $(\|x_n-z\|^2)$ we have that $q_v \leq q_z + \frac{u'}2$.

By the above applied to $z:=\frac{v+h_Tv}2$ and $q_z$ an $\frac{u'}4$-limsup of $(\|x_n-z\|^2)$, we get after using Corollary~\ref{c1} that there is an $h \geq  \alpha\left(\left\lceil\max\left\{\frac{2b}{\sqrt{\nu_1(w,k,k',\eps)}},\frac{8b^2}{\nu_1(w,k,k',\eps)}\right\}\right\rceil\right)$ such that
$$\|x_h-v\|^2 - \left\|x_h-\frac{v+h_Tv}2\right\|^2 \leq u' \leq \frac{\nu_1(w,k,k',\eps)}2.$$

\item The construction of $l$.

Since $q_w$ is a $\frac{u}4$-limsup of $(\|x_n-w\|^2)$, it is also a $\frac{\nu_2(\eps)}4$-limsup of $(\|x_n-w\|^2)$. Similarly, $q_v$ is a $\frac{\nu_2(\eps)}4$-limsup of $(\|x_n-v\|^2)$.

Put $z:=\frac{v+w}2$ and take $q_z$ to be a $\frac{\min\{u,u'\}}4$-limsup of $(\|x_n-z\|^2)$.

Since $q_z$ is also a $\frac{u}4$-limsup of $(\|x_n-z\|^2)$,
$$q_w\leq q_z + \frac{u}2 \leq q_z + \frac{\nu_2(\eps)}2$$
and similarly
$$q_v\leq q_z + \frac{\nu_2(\eps)}2.$$
Also take note that $q_z$ is a $\frac{\nu_2(\eps)}4$-limsup of $(\|x_n-z\|^2)$. By Corollary~\ref{c2}, we get that there is an $l$ such that
$$\|x_l-w\|^2 - \left\|x_l -\frac{v+w}2\right\|^2,\quad\|x_l-v\|^2 - \left\|x_l -\frac{v+w}2\right\|^2 \leq \nu_2(\eps).$$

\item The construction of $v'$ and $h'$.

We apply Claim I for $(t_{s_{w,g}(n)})$ and $u'$. We obtain $v' \in C$ and $q_{v'} \in \Q$ such that $q_{v'}$ is an $\frac{u'}4$-limsup of $(\|x^w_n-v'\|^2)$ and for all $z \in C$ and $q_z \in \Q$ with $q_z$ being an $\frac{u'}4$-limsup of $(\|x^w_n-z\|^2)$ we have that $q_{v'} \leq q_z + \frac{u'}2$.

By the above applied to $z:=\frac{v'+h_Tv'}2$ and $q_z$ an $\frac{u'}4$-limsup of $(\|x^w_n-z\|^2)$, we get after using Corollary~\ref{c1} that there is an $h'_0 \geq \alpha\left(\left\lceil\max\left\{\frac{2b}{\sqrt{\nu_1(w,k,k',\eps)}},\frac{8b^2}{\nu_1(w,k,k',\eps)}\right\}\right\rceil\right)$ such that
$$\|x^w_{h'_0}-v'\|^2 - \left\|x^w_{h'_0}-\frac{v'+h_Tv'}2\right\|^2 \leq u' \leq \frac{\nu_1(w,k,k',\eps)}2.$$
Put $h':=s_{w,g}(h'_0) \geq h'_0 \geq  \alpha\left(\max\left\{\frac{2b}{\sqrt{\nu_1(w,k,k',\eps)}},\frac{8b^2}{\nu_1(w,k,k',\eps)}\right\}\right)$. Then
$$\|x_{h'}-v'\|^2 - \left\|x_{h'} -\frac{v'+h_Tv'}2\right\|^2 \leq \frac{\nu_1(w,k,k',\eps)}2.$$

\item The construction of $l'$.

Since $q_{w'}$ is a $\frac{u}4$-limsup of $(\|x^w_n-w'\|^2)$, it is also a $\frac{\nu_2(\eps)}4$-limsup of $(\|x^w_n-w'\|^2)$. Similarly, $q_{v'}$ is a $\frac{\nu_2(\eps)}4$-limsup of $(\|x^w_n-v'\|^2)$.

Put $z:=\frac{v'+w'}2$ and take $q_z$ to be a $\frac{\min\{u,u'\}}4$-limsup of $(\|x^w_n-z\|^2)$.

Since $q_z$ is also a $\frac{u}4$-limsup of $(\|x^w_n-z\|^2)$,
$$q_{w'}\leq q_z + \frac{u}2 \leq q_z + \frac{\nu_2(\eps)}2$$
and similarly
$$q_{v'}\leq q_z + \frac{\nu_2(\eps)}2.$$
Also take note that $q_z$ is a $\frac{\nu_2(\eps)}4$-limsup of $(\|x^w_n-z\|^2)$. By Corollary~\ref{c2}, we get that there is an $l'_0$ such that
$$\|x^w_{l'_0}-w'\|^2 - \left\|x^w_{l'_0} -\frac{v'+w'}2\right\|^2,\quad\|x^w_{l'_0}-v'\|^2 - \left\|x^w_{l'_0} -\frac{v'+w'}2\right\|^2 \leq \nu_2(\eps).$$
Put $l':=s_{w,g}(l'_0)$. Then
$$\|x_{l'}-w'\|^2 - \left\|x_{l'} -\frac{v'+w'}2\right\|^2,\quad\|x_{l'}-v'\|^2 - \left\|x_{l'} -\frac{v'+w'}2\right\|^2 \leq \nu_2(\eps).$$
\end{enumerate}

We are now done.\hfill $\blacksquare$\\[2mm]

\noindent {\bf Claim III.} Let $w$, $w'$, $v$, $v'\in C$ and $k$, $k'$, $l$, $l'$, $h$, $h' \in \N$ be as in Claim II. If we put $N:=k'$, then $\|x_N - x_{N+g(N)}\| \leq \eps$. (One notices here that the part of the proof corresponding to this claim will not need further inspection, as it is enough to obtain a value for $k'$ in an analysis of Claim II.)\\[1mm]
{\bf Proof of claim:} We will further divide the proof of this claim into sub-claims.\\[1mm]

\noindent {\bf Sub-claim 1.} We have that:
\begin{itemize}
\item $\|x_{h}-v\|^2 - \left\|x_{h} -\frac{v+h_Tv}2\right\|^2$, $\|x_{h}-h_Tv\|^2 - \left\|x_{h} -\frac{v+h_Tv}2\right\|^2$,\\
 $\|x_{h'}-v'\|^2 - \left\|x_{h'} -\frac{v'+h_Tv'}2\right\|^2$, $\|x_{h'}-h_Tv'\|^2 - \left\|x_{h'} -\frac{v'+h_Tv'}2\right\|^2  \leq \nu_1(w,k,k',\eps)$.
\end{itemize}
{\bf Proof of sub-claim:} First, we remark that:
$$\|x_h - Tx_h\| = \left\|t_hTx_h + (1-t_h)x - Tx_h\right\| = \left\|(1-t_h)(x-Tx_h) \right\| \leq b(1-t_h),$$
so, using Corollary~\ref{cor-ht}.(ii),
$$\|x_h-h_Tx_h\| \leq b(1-t_h)$$
and, by Corollary~\ref{cor-ht}.(i),
$$\|x_h-h_Tv\| \leq \|h_Tv-h_Tx_h\| + \|x_h-h_Tx_h\| \leq \|x_h - v\| + b(1-t_h).$$
Now we may write:
\begin{align*}
\|x_h - h_Tv\|^2 &\leq \|x_h-v\| ^ 2 +b^2(1-t_h)^2+ 2b\cdot b (1-t_h)\\
&\leq \left\|x_h - \frac{v+h_Tv}2 \right\|^2 + \frac{\nu_1(w,k,k',\eps)}2 + \frac{b^2}{\frac{4b^2}{\nu_1(w,k,k',\eps)}} + \frac{2b^2}{\frac{8b^2}{\nu_1(w,k,k',\eps)}} \\
&\leq \left\|x_h - \frac{v+h_Tv}2 \right\|^2 + \frac{\nu_1(w,k,k',\eps)}2 + \frac{\nu_1(w,k,k',\eps)}4 + \frac{\nu_1(w,k,k',\eps)}4 \\
&=  \left\|x_h - \frac{v+h_Tv}2 \right\|^2 + \nu_1(w,k,k',\eps).
\end{align*}
Similarly, one may show that $\|x_{h'}-h_Tv'\|^2 - \left\|x_{h'} -\frac{v'+h_Tv'}2\right\|^2  \leq \nu_1(w,k,k',\eps)$.\hfill $\blacksquare$\\[2mm]

\noindent {\bf Sub-claim 2.} We have that:
$$\|h_Tv-v\|,\ \|h_Tv'-v'\| \leq \tilde{\theta}(q(w,k,k',\eps)).$$
{\bf Proof of sub-claim:} Suppose that $\|h_Tv-v\| \geq \tilde{\theta}(q(w,k,k',\eps))$. Then
$$\|(x_h-v)-(x_h-h_Tv)\| = \|h_Tv-v\| \geq  \tilde{\theta}(q(w,k,k',\eps)),$$
so
\begin{align*}
\left\|x_h - \frac{v+h_Tv}2\right\|^2 + \psi_{b,\eta}(\tilde{\theta}(q(w,k,k',\eps))) &\leq \frac12\|x_h-v\|^2 + \frac12\|x_h-h_Tv\|^2 \\
&\leq \left\|x_h-\frac{v+h_Tv}2\right\|^2 + \frac12\psi_{b,\eta}(\tilde{\theta}(q(w,k,k',\eps))),
\end{align*}
which is a contradiction, since $\psi_{b,\eta}(\tilde{\theta}(q(w,k,k',\eps)))  >0$.

Similarly, one shows that $\|h_Tv'-v'\| \leq \tilde{\theta}(q(w,k,k',\eps))$.\hfill $\blacksquare$\\[2mm]

\noindent {\bf Sub-claim 3.}  We have that:
\begin{itemize}
\item $\langle x_k-x, j(x_k-w) \rangle$, $\langle x^w_{k'}-x,j(x^w_{k'}-w')\rangle \leq \frac{\nu_4(\eps)}3$;
\item $\langle x^w_{k'}-x,j(x^w_{k'}-w)\rangle$, $\langle x_k-x, j(x_k-w') \rangle \leq \frac{\eps^2}{96}$.
\end{itemize}
{\bf Proof of sub-claim:} Since $\|h_Tv-v\| \leq \tilde{\theta}(q(w,k,k',\eps))$, we have, using Corollary~\ref{cor-ht}.(iii), that $\|Tv-v\| \leq q(w,k,k',\eps)$. We now compute:
\begin{align*}
\langle x_k-x, j(x_k-v) \rangle &= \left\langle t_k(Tx_k-x), j(x_k-v) \right\rangle \\
&=t_k\langle Tx_k -Tv, j(x_k-v) \rangle + t_k \langle Tv-v, j(x_k -v) \rangle + t_k\langle v-x, j(x_k - v) \rangle \\
&\leq t_k\|x_k-v\|^2  + \|Tv-v\| \|x_k-v\| + t_k\langle v-x, j(x_k - v) \rangle \\
&\leq t_k\|x_k-v\|^2  + t_k\langle v-x, j(x_k - v) \rangle + b\|Tv-v\| \\
&\leq t_k\langle x_k-v, j(x_k-v) \rangle + t_k \langle v-x,j(x_k -v) \rangle + b\|Tv-v\| \\
&= t_k \langle x_k - x, j(x_k - v) \rangle +b\|Tv-v\| ,
\end{align*}
from which we obtain
\begin{align*}
\langle x_k-x, j(x_k-v) \rangle &\leq \frac{b}{1-t_k}\|Tv-v\|\\
&\leq \frac{b}{1-t_k} \cdot q(w,k,k',\eps) \\
&\leq \frac{b}{1-t_k} \cdot \beta(k,\eps) = \frac{b}{1-t_k} \cdot \frac{p(\eps)}{2b\gamma(k)} \leq \frac{p(\eps)}2.
\end{align*}
Suppose now that $\|w-v\| \geq \omega_\tau\left(b,\frac{p(\eps)}{2b}\right)$. Then
$$\|(x_l - w) - (x_l - v)\| = \|w-v\| \geq \omega_\tau\left(b,\frac{p(\eps)}{2b}\right)$$
and so
\begin{align*}
\left\|x_l - \frac{v+w}2\right\|^2 + \psi_{b,\eta}\left(\omega_\tau\left(b,\frac{p(\eps)}{2b}\right)\right) &\leq \frac12\|x_l-v\|^2 + \frac12\|x_l-w\|^2 \\
&\leq \left\|x_l-\frac{v+w}2\right\|^2 + \frac12\psi_{b,\eta}\left(\omega_\tau\left(b,\frac{p(\eps)}{2b}\right)\right),
\end{align*}
which is a contradiction, since $\psi_{b,\eta}\left(\omega_\tau\left(b,\frac{p(\eps)}{2b}\right)\right)  >0$.

Therefore
$$\|(x_k-w)-(x_k-v)\| = \|w-v\| \leq \omega_\tau\left(b,\frac{p(\eps)}{2b}\right),$$
so
$$\|j(x_k-w)-j(x_k-v)\| \leq \frac{p(\eps)}{2b}.$$
We have then
\begin{align*}
\langle x_k-x, j(x_k-w) \rangle&\leq \langle x_k-x, j(x_k-v) \rangle + \langle  x_k-x,j(x_k-w)-j(x_k-v)\rangle\\
& \leq  \frac{p(\eps)}2 +  b\cdot\frac{p(\eps)}{2b} = p(\eps).
\end{align*}

Similarly, taking into account, when needed, that
$$q(w,k,k',\eps)\leq\beta(s_{w,g}(k'),\eps),$$
we obtain the other three inequalities.\hfill $\blacksquare$\\[2mm]

\noindent {\bf Sub-claim 4.} We have that:
$$\|x_k-w\|^2,\ \|x^w_{k'}-w'\|^2 \leq \nu_4(\eps).$$
{\bf Proof of sub-claim:} By Lemma~\ref{jl}, we have:
$$\|x_k - w - \delta(\eps)(x-w) \|^2 \leq \|x_k-w\|^2 + 2\langle - \delta(\eps)(x-w), j(x_k-w-\delta(\eps)(x-w)) \rangle.$$
Given that, as before, $\delta(\eps) \in (0,1)$ and so $w+\delta(\eps)(x-w) = (1-\delta(\eps))w+\delta(\eps)x \in C$, and that
$$\|(x_k-w) - (x_k-w-\delta(\eps)(x-w))\| = \|\delta(\eps)(x-w)\| \leq \delta(\eps) \cdot b \leq \omega_\tau\left(b,\frac{\nu_4(\eps)}{3b} \right),$$
we get that
$$\|j(x_k-w)-j(x_k-w-\delta(\eps)(x-w))\| \leq \frac{\nu_4(\eps)}{3b}.$$
Therefore (using the first item in Claim II),
\begin{align*}
\|x_k-w\|^2 &= \langle x_k-w, j(x_k-w) \rangle \\
&\leq \langle x_k - x, j(x_k-w) \rangle + \langle x -w, j(x_k-w)-j(x_k-w-\delta(\eps)(x-w) )\rangle\\
&\quad + \langle x - w, j(x_k-w-\delta(\eps)(x-w))  \rangle \\
&\leq \frac{\nu_4(\eps)}3 + b \cdot \frac{\nu_4(\eps)}{3b} + \frac1{2\delta(\eps)}(\|x_k-w\|^2 - \|x_k - w - \delta(\eps)(x-w) \|^2 ) \\
&\leq \frac{2\nu_4(\eps)}3 +  \frac1{2\delta(\eps)} \cdot  \frac{2\nu_4(\eps)\delta(\eps)}3 = \nu_4(\eps).
\end{align*}

In a similar way, using the fact that $\langle x^w_{k'}-x,j(x^w_{k'}-w')\rangle \leq \frac{\nu_4(\eps)}3$, we obtain the other inequality to be proven.\hfill $\blacksquare$\\[2mm]

\noindent {\bf Sub-claim 5.} We have that $\|x^w_N - w\| \leq \frac\eps2$.\\[1mm]
{\bf Proof of sub-claim:} We know that $\|x^w_{k'} - w' \|\leq \frac{\eps^2}{96b}$. Since
$$\|(x^w_{k'}-w) - (w'-w)\| = \|x^w_{k'} - w'\| \leq \omega_\tau\left(b,\frac{\eps^2}{96b}\right),$$
we have that
$$\|j(x^w_{k'}-w)-j(w'-w)\| \leq \frac{\eps^2}{96b}$$
and so
\begin{align*}
\langle w'-x, j(w'-w) \rangle &\leq \langle w'-x^w_{k'},j(w'-w) \rangle + \langle x^w_{k'}-x, j(w'-w) - j(x^w_{k'} -w ) \rangle \\
&\quad+ \langle x^w_{k'} - x, j(x^w_{k'}-w) \rangle \\
&\leq b \cdot \frac{\eps^2}{96b} + b \cdot \frac{\eps^2}{96b}  + \frac{\eps^2}{96}  = \frac{\eps^2}{32}.
\end{align*}
Similarly, using $x_k$ as the ``pivot'', we get that $\langle w-x, j(w-w') \rangle \leq \frac{\eps^2}{32}$, so
$$\langle w-w',j(w-w') \rangle \leq \frac{\eps^2}{16},$$
i.e.
$$\|w-w'\| \leq \frac\eps4.$$
We can now compute:
$$\|x^w_N - w\| = \|x^w_{k'} - w\| \leq \|x^w_{k'} - w'\| + \|w-w'\| \leq \frac\eps4 + \frac\eps4 = \frac\eps2,$$
which is what we wanted.
\hfill $\blacksquare$\\[2mm]

It follows immediately, by the definition of $x^w_N$, that $\max\{\|x_N - w\|, \|x_{N+g(N)} - w\| \}\leq \frac\eps2 $. To finish the proof of the claim, we see that
$$\|x_N - x_{N+g(N)}\| \leq \|x_N - w\| + \|x_{N+g(N)} - w\| \leq \frac\eps2+\frac\eps2 = \eps,$$
which also finishes the proof of the theorem.\hfill $\blacksquare$\\[2mm]

\section{The extraction of the witness}\label{sec:proof-mining}

\subsection{The logical analysis of Claim I}

The first proposition
in this section, Proposition~\ref{exlimsup-nd}, is the (partial) 
functional interpretation of 
Proposition~\ref{exlimsup}, i.e. of the existence of $\varepsilon$-limsup's 
using only functionals definable in the fragment 
$T_1$ (which only contains the 
recursor constants $R_0$ and $R_1$) of G\"odel's $T$. This analysis was obtained with the crucial guidance of the functional interpretation of induction from \cite{Par72}.

We then give in Proposition 
\ref{claim-i} the (partial) functional interpretation of the proof 
of Claim I, i.e. the existence of $\varepsilon$-infima for approximate 
limsup's, by functions definable in $T_2$ (as now also $R_2$ is used).
Since the functional interpretation of the Claim II, which only uses 
the existence of approximate limsup's and plain logic plus elementary 
arithmetic, can be interpreted already in 
$T_1,$ this guarantees the extractability of a rate of metastability 
definable in $T_2.$

In the following we use, for any $n,m \in \N$, the notation $n \dotdiv m$ to denote $n-m$ if $n \geq m$ and $0$ otherwise. We also use the usual conventions 
when defining higher-order functionals and write e.g. `$JUM(b\cdot (k+1)
\dotdiv PUM)$' instead of `$J(U,M,b\cdot (k+1)\dotdiv P(U,M))$'. Occasionally, 
we also use the $\lambda$-notation $\lambda x_1,\ldots,x_n.t[x_1,\ldots,x_n]$ 
from functional programming, for a term $t[x_1,\ldots,x_n]$ depending on the variables $x_1,\ldots,x_n,$ to 
denote the function: $(x_1,\ldots,x_n)\mapsto t[x_1,\ldots,x_n].$

\begin{proposition}\label{exlimsup-nd}
Let $b, k \in \N$ and $(a_n)$ be a sequence of reals contained in the interval $[0,b]$. Define the following functionals:
\begin{align*}
WUM0&:=_{\N^\N} 0^1. \\
WUM(n+1)&:=_{\N^\N} \lambda y.U(WUMn,y,n). \\
JUM0 &:=_{\N} 0. \\
JUM(n+1) &:=_{\N} M(WUM(b\cdot (k+1)  \dotdiv n), JUMn, b\cdot (k+1) \dotdiv n). \\
PUM &:=_{\N}
  \begin{cases} 
       \textup{the least natural number $p \leq b \cdot (k+1)$ such that}\\
       \quad a_{JUM(b\cdot (k+1) + 1 \dotdiv p) + WUMp(JUM(b\cdot (k+1) + 1 \dotdiv p))} > \frac{p-1}{k+1} \\
        \quad \textup{and } a_{JUM(b\cdot (k+1)  \dotdiv p) + WUM(p+1)(JUM(b\cdot (k+1) \dotdiv p))} \leq \frac{p}{k+1}, \hfill & \textup{if there is one}, \\
       0, \hfill & \textup{otherwise.} \\
  \end{cases}\\
NUM &:=_{\N^\N} WUM(PUM). \\
TUM &:=_{\N} JUM(b \cdot (k+1) \dotdiv PUM).
\end{align*}

Then, for all $U,M : \N^\N \times \N \times \N \to \N$, we have that $0 \leq PUM \leq b\cdot (k+1)$ and:
\begin{enumerate}[(i)]
\item $a_{M(NUM,TUM,PUM) + (NUM)(M(NUM,TUM,PUM))} \geq \frac{PUM}{k+1} - \frac1{k+1}$;
\item $a_{TUM + U(NUM,TUM,PUM)} \leq \frac{PUM}{k+1} + \frac1{k+1}$.
\end{enumerate}
\end{proposition}

\begin{proof}
We start with the following claim, analogous to the one in the proof of Proposition~\ref{exlimsup}.\\[1mm]

\noindent {\bf Claim.} There is a $p \in \N$ with $0 \leq p \leq b\cdot (k+1)$ such that it is not the case that
$$a_{JUM(b\cdot (k+1) + 1 \dotdiv p) + WUMp(JUM(b\cdot (k+1) + 1 \dotdiv p))} > \frac{p-1}{k+1}$$
implies that 
$$a_{JUM(b\cdot (k+1) \dotdiv p) + WUM(p+1)(JUM(b\cdot (k+1) \dotdiv p))} > \frac{p}{k+1}.$$
{\bf Proof of claim:} Assume towards a contradiction that the opposite holds, i.e. for all natural numbers $p$ smaller or equal to $b \cdot (k+1)$, we have that $Q(p)$ implies $Q(p+1)$, where $Q(p)$ is the statement that
$$a_{JUM(b\cdot (k+1)+ 1 \dotdiv p) + WUMp(JUM(b\cdot (k+1) + 1 \dotdiv p))} > \frac{p-1}{k+1}.$$
Since $Q(0)$ holds trivially, we have by induction that $Q(b \cdot (k+1) +1)$. But that states that
$$a_{JUM0 + WUM(b \cdot (k+1)+ 1)(JUM0)} > b,$$
which clearly is false.
\hfill $\blacksquare$\\[2mm]

Take $p$ to be minimal with this property. Then $p=PUM$, by the definition of the latter, so clearly $0 \leq PUM \leq b\cdot (k+1)$. We prove the remaining conclusions.
\begin{enumerate}[(i)]
\item Since $PUM < b\cdot (k+1) + 1$, we may write $b\cdot (k+1) + 1 \dotdiv PUM = (b\cdot (k+1) \dotdiv PUM) + 1$, so
\begin{align*}
JUM(b\cdot (k+1) + 1  \dotdiv PUM) &= JUM((b\cdot (k+1) \dotdiv PUM) + 1) \\
&= M(WUM(PUM), JUM(b\cdot (k+1) \dotdiv PUM), PUM) \\
&= M(NUM, TUM, PUM)
\end{align*}
and
$$WUM(PUM)(JUM(b\cdot (k+1) + 1 \dotdiv PUM)) = NUM(M(NUM, TUM, PUM)).$$
Thus,
$$a_{M(NUM,TUM,PUM) + (NUM)(M(NUM,TUM,PUM))} > \frac{PUM-1}{k+1} = \frac{PUM}{k+1} - \frac1{k+1}.$$
\item Since
$$JUM(b\cdot (k+1)  \dotdiv PUM) = TUM$$
and
\begin{align*}
WUM(PUM+1)(JUM(b\cdot (k+1) \dotdiv PUM))&= WUM(PUM+1)(TUM) \\
&= U(WUM(PUM),TUM,PUM) \\
&= U(NUM,TUM,PUM),
\end{align*}
we have that
$$a_{TUM + U(NUM,TUM,PUM)} \leq \frac{PUM}{k+1}\leq \frac{PUM}{k+1} + \frac1{k+1}.$$
\end{enumerate}
The proof is finished.
\end{proof}

The following is a logical analysis of Claim I in the second proof (using approximate limsup's) of Theorem~\ref{main-thm} and uses as an ingredient the functional interpretation of the existence of approximate limsup's. Here, and in the remainder of the paper, we shall frequently use numerical indices to refer to components of tuples.

\begin{proposition}\label{claim-i}
Let $b \in \N^*$. Let $X$ be a Banach space, $C \subseteq X$ be a set of diameter at most $b$ and $(x_n) \subseteq C$. Let $\eps >0$. Let $z_1 \in C$ be arbitrary. Define the following functionals (where any $\mathcal{O}$ denotes a constant zero function):
\begin{align*}
(M(\Omega,0),U(\Omega,0))&:=(\mathcal{O},\mathcal{O}).\\
(M(\Omega,x+1),U(\Omega,x+1))&:=\lambda p,y,L,m. (\Omega_{5,6}(M(\Omega,x),U(\Omega,x),p,y,L,m)).\\
I&:=\left\lceil\frac{2b^2}{\eps} \right\rceil.\\
\overline{M}(\Omega)&:=\lambda L,m,p. ((M(\Omega,I+1))(p,z_1,L,m)). \\
\overline{U}(\Omega)&:=\lambda L,m,p. ((U(\Omega,I+1))(p,z_1,L,m)). \\
\WPsi(\Omega,0)&:=(M(\Omega,I), U(\Omega,I),P(\overline{U}(\Omega),\overline{M}(\Omega)),z_1,N(\overline{U}(\Omega),\overline{M}(\Omega)), T(\overline{U}(\Omega),\overline{M}(\Omega))). \\
\WPsi(\Omega,x+1) &:= (M(\Omega,I\dotdiv (x+1)), U(\Omega, I\dotdiv (x+1)), \Omega_{1-4}(\WPsi(\Omega,x))).
\end{align*}
In the above, $P$, $N$ and $T$ are the functionals defined in Proposition~\ref{exlimsup-nd}, customized by instantiating their free parameters with $b \mapsto b^2$, $k \mapsto \left\lceil\frac{4}{\eps}\right\rceil$ and $(a_n) \mapsto (\|x_n-z_1\|^2)$. We continue to use in the following the notation $k := \left\lceil\frac{4}{\eps}\right\rceil$.

Then, for any $\Omega$ there is an $i < I$ such that if we denote
$$(\widetilde{U},\widetilde{N},p,y,L,m):=\WPsi(\Omega,i)$$
and
$$(r,z,\widetilde{L},\widetilde{m},u,n):=\Omega(\WPsi(\Omega,i)),$$
we have that
$$0 \leq p \leq b^2 \cdot (k+1),\ \|x_{u+L(u)}-y\|^2 \geq \frac{p}{k+1} - \frac\eps4  \text{ and } \|x_{m+n}-y\|^2 \leq \frac{p}{k+1} + \frac\eps4$$
and that if
$$0 \leq r \leq b^2 \cdot (k+1),\ \|x_{\widetilde{U}(r,z,\widetilde{L},\widetilde{m})+\widetilde{L}(\widetilde{U}(r,z,\widetilde{L},\widetilde{m}))}-z\|^2 \geq \frac{r}{k+1} - \frac\eps4 \text{ and } \|x_{\widetilde{m}+\widetilde{N}(r,z,\widetilde{L},\widetilde{m})}-z\|^2 \leq \frac{r}{k+1} + \frac\eps4$$
then
$$\frac{p}{k+1} \leq \frac{r}{k+1} + \frac\eps2.$$

In order to obtain a true realizer, we now put, for any $\Omega$, $i(\Omega)$ to be the least $i < I$ which realizes the above (in order to define it properly as a functional, we put it to be $0$ in the ``impossible'' case that there isn't one, as in the definition of $P$ in Proposition~\ref{exlimsup-nd}) and $\Psi(\Omega)$ to be $\WPsi(\Omega,i(\Omega))$.
\end{proposition}

\begin{proof}
Assume towards a contradiction that the opposite holds, i.e. there is an $\Omega$ such that if we denote for all $x\leq I$,
$$(\widetilde{U}_{I-x},\widetilde{N}_{I-x},p_{I-x},y_{I-x},L_{I-x},m_{I-x}):=\WPsi(\Omega,x)$$
and
$$(r_{I-x},z_{I-x},\widetilde{L}_{I-x},\widetilde{m}_{I-x},u_{I-x},n_{I-x}):=\Omega(\WPsi(\Omega,x)),$$
then for all $x < I$, if
$$0 \leq p_{I-x} \leq b^2 \cdot (k+1),\ \|x_{u_{I-x}+L_{I-x}(u_{I-x})}-y_{I-x}\|^2 \geq \frac{p_{I-x}}{k+1} - \frac\eps4 \text{ and } \|x_{m_{I-x}+n_{I-x}}-y_{I-x}\|^2 \leq \frac{p_{I-x}}{k+1} + \frac\eps4$$
then
$$0 \leq r_{I-x} \leq b^2 \cdot (k+1),$$
$$\|x_{\widetilde{U}_{I-x}(r_{I-x},z_{I-x},\widetilde{L}_{I-x},\widetilde{m}_{I-x})+\widetilde{L}_{I-x}(\widetilde{U}_{I-x}(r_{I-x},z_{I-x},\widetilde{L}_{I-x},\widetilde{m}_{I-x}))}-z_{I-x}\|^2 \geq \frac{r_{I-x}}{k+1} - \frac\eps4,$$
$$\|x_{\widetilde{m}_{I-x}+\widetilde{N}_{I-x}(r_{I-x},z_{I-x},\widetilde{L}_{I-x},\widetilde{m}_{I-x})}-z_{I-x}\|^2 \leq \frac{r_{I-x}}{k+1} + \frac\eps4$$
and
$$\frac{r_{I-x}}{k+1} < \frac{p_{I-x}}{k+1} - \frac\eps2.$$
Remark that for all $x \leq I$,
$$(\widetilde{U}_{I-x}, \widetilde{N}_{I-x}) = (\WPsi(\Omega,x))_{1,2} = (M(\Omega,I-x),U(\Omega,I-x)).$$
In addition,
$$p_I=P(\overline{U}(\Omega),\overline{M}(\Omega)),\quad y_I=z_1,\quad L_I=N(\overline{U}(\Omega),\overline{M}(\Omega)),\quad m_I=T(\overline{U}(\Omega),\overline{M}(\Omega)).$$
We now derive that for all $x < I$,
\begin{align*}
(r_{I-x},z_{I-x},\widetilde{L}_{I-x},\widetilde{m}_{I-x}) &= \Omega_{1-4}(\WPsi(\Omega,x)) \\
&= (\WPsi(\Omega,x+1))_{3-6} \\
&= (p_{I-x-1},y_{I-x-1},L_{I-x-1},m_{I-x-1})
\end{align*}
and
\begin{align*}
\widetilde{U}_{I-x}(r_{I-x},z_{I-x},\widetilde{L}_{I-x},\widetilde{m}_{I-x}) &= \widetilde{U}_{I-x} (p_{I-x-1},y_{I-x-1},L_{I-x-1},m_{I-x-1})\\
&= \Omega_5(\WPsi(\Omega,x+1)) \\
&= (\Omega(\WPsi(\Omega,x+1)))_5 \\
&= u_{I-x-1},
\end{align*}
together with the corresponding statement $\widetilde{N}_{I-x}(r_{I-x},z_{I-x},\widetilde{L}_{I-x},\widetilde{m}_{I-x})  = n_{I-x-1}$. Therefore, what we know is that for all $x < I$, if
$$0 \leq p_{I-x} \leq b^2 \cdot (k+1),\ \|x_{u_{I-x}+L_{I-x}(u_{I-x})}-y_{I-x}\|^2 \geq \frac{p_{I-x}}{k+1} - \frac\eps4 \text{ and } \|x_{m_{I-x}+n_{I-x}}-y_{I-x}\|^2 \leq \frac{p_{I-x}}{k+1} + \frac\eps4$$
then
$$0 \leq p_{I-x-1} \leq b^2 \cdot (k+1),$$
$$\|x_{u_{I-x-1}+L_{I-x-1}(u_{I-x-1})}-y_{I-x-1}\|^2 \geq \frac{p_{I-x-1}}{k+1} - \frac\eps4,$$
$$\|x_{m_{I-x-1}+n_{I-x-1}}-y_{I-x-1}\|^2 \leq \frac{p_{I-x-1}}{k+1} + \frac\eps4$$
and
$$\frac{p_{I-x-1}}{k+1} < \frac{p_{I-x}}{k+1} - \frac\eps2.$$
We shall now prove by induction that for all $x \leq I$,
$$0 \leq p_{I-x} \leq b^2 \cdot (k+1),\ \|x_{u_{I-x}+L_{I-x}(u_{I-x})}-y_{I-x}\|^2 \geq \frac{p_{I-x}}{k+1} - \frac\eps4 \text{ and } \|x_{m_{I-x}+n_{I-x}}-y_{I-x}\|^2 \leq \frac{p_{I-x}}{k+1} + \frac\eps4.$$
It remains to show the base case ($x=0$): we apply Proposition~\ref{exlimsup-nd} for $\overline{U}(\Omega)$ and $\overline{M}(\Omega)$. Using
\begin{align*}
\overline{M}(\Omega)(N(\overline{U}(\Omega),\overline{M}(\Omega)),T(\overline{U}(\Omega),\overline{M}(\Omega)),P(\overline{U}(\Omega),\overline{M}(\Omega))) &= \overline{M}(\Omega)(L_I,m_I,p_I) \\
&= M(\Omega, I+1)(p_I,y_I,L_I,m_I) \\
&= \Omega_5(M(\Omega,I),U(\Omega,I),p_I,y_I,L_I,m_I) \\
&= \Omega_5(\widetilde{U}_I, \widetilde{N}_I,p_I,y_I,L_I,m_I) \\
&= \Omega_5(\WPsi(\Omega,0)) \\
&= u_I
\end{align*}
and - similarly -
$$\overline{U}(\Omega)(N(\overline{U}(\Omega),\overline{M}(\Omega)),T(\overline{U}(\Omega),\overline{M}(\Omega)),P(\overline{U}(\Omega),\overline{M}(\Omega)))
=n_I,$$ 
we see that what we obtain is that
$$0 \leq p_{I} \leq b^2 \cdot (k+1),\ \|x_{u_{I}+L_{I}(u_{I})}-y_{I}\|^2 \geq \frac{p_{I}}{k+1} - \frac1{k+1} \text{ and } \|x_{m_{I}+n_{I}}-y_{I}\|^2 \leq \frac{p_{I}}{k+1} + \frac1{k+1}.$$
Using that $\frac1{k+1} \leq \frac\eps4$, we obtain that
$$0 \leq p_{I} \leq b^2 \cdot (k+1),\ \|x_{u_{I}+L_{I}(u_{I})}-y_{I}\|^2 \geq \frac{p_{I}}{k+1} - \frac\eps4 \text{ and } \|x_{m_{I}+n_{I}}-y_{I}\|^2 \leq \frac{p_{I}}{k+1} + \frac\eps4,$$
i.e. what we wanted. The induction case follows immediately by our assumption. Therefore we have that for all $x <I$,
$$\frac{p_{I-x-1}}{k+1} < \frac{p_{I-x}}{k+1} - \frac\eps2,$$
so
$$\frac{p_0}{k+1} < \frac{p_{I}}{k+1} - I\cdot\frac\eps2.$$
Since by construction $p_I \leq b^2 \cdot (k+1)$ and $0 \leq \frac{p_0}{k+1}$, what we obtain is
$$0 < b^2 - I\cdot\frac\eps2,$$
contradicting the fact that $I=\left\lceil\frac{2b^2}{\eps} \right\rceil$.
\end{proof}

\subsection{The logical analysis of Claim II}

In the sequel, we shall denote by (here $x$ stands for the sequence $(x_n)$)
$$A(x,\delta,\widetilde{U},\widetilde{N},p,y,L,r,z,\widetilde{L},\widetilde{m},u,t)$$
the statement that (where we write $k := \left\lceil\frac{4}{\delta}\right\rceil$)
$$0 \leq p \leq b^2 \cdot (k+1),\ \|x_{u+L(u)}-y\|^2 \geq \frac{p}{k+1} - \frac\delta4  \text{ and } \|x_t-y\|^2 \leq \frac{p}{k+1} + \frac\delta4$$
and that if
$$0 \leq r \leq b^2 \cdot (k+1),\ \|x_{\widetilde{U}(r,z,\widetilde{L},\widetilde{m})+\widetilde{L}(\widetilde{U}(r,z,\widetilde{L},\widetilde{m}))}-z\|^2 \geq \frac{r}{k+1} - \frac\delta4 \text{ and } \|x_{\widetilde{m}+\widetilde{N}(r,z,\widetilde{L},\widetilde{m})}-z\|^2 \leq \frac{r}{k+1} + \frac\delta4$$
then
$$\frac{p}{k+1} \leq \frac{r}{k+1} + \frac\delta2.$$

We will also make the parameters $(x_n)$ and $\eps >0$ (though not $b$) in the $\Psi$ from Proposition~\ref{claim-i} explicit in what follows. Thus, 
Proposition~\ref{claim-i} states that for any $(x_n) \subseteq C$, $\eps > 0$, $g$ and $f$, if we put
$$(\uu{w},k):=\Psi(x,\eps,(g,f)),$$
i.e., in particular, $\uu{w}$ is a 5-tuple -- corresponding to $(\widetilde{U},\widetilde{N},p,y,L)$ in the above definition -- whereas $k$ is a number (and also, we add for clarity, $g$ returns a 5-tuple -- corresponding to $(r,z,\widetilde{L},\widetilde{m},u)$ -- while $f$ returns a number), then
$$A(x,\eps,\uu{w},g(\uu{w},k),k+f(\uu{w},k)).$$

\subsubsection{Preparation}

In the proposition below, the eight items correspond to the eight inequalities involving the sequence $(x_n)$ that must be satisfied in Claim II.

\begin{proposition}\label{lambda-term}
Let $X$ be a Banach space, $b \in \N^*$ and $C \subseteq X$ be a set of diameter at most $b$. Let $(x_n) \subseteq C$. Then there is a $\Phi$ such that for any $\Psi$ having the property that for any $(y_n) \subseteq C$, $\delta > 0$, $g$ and $f$, if we put
$$(\uu{w},k):=\Psi(y,\delta,(g,f))$$
then one has
$$A(y,\delta,\uu{w},g(\uu{w},k),k+f(\uu{w},k)),$$
we have that for any $u$, $u'$, $\uu{g}$, $\uu{g}'$, $\uu{f}$, $\uu{f}'$, $\iota$, and $\vp$, if we put
$$(\uu{w},k,\uu{w}',k',\uu{v},\uu{v}',l,l',h,h'):=\Phi(\Psi)(u, u', \uu{g}, \uu{g'}, \uu{f}, \uu{f}', \iota, \vp)$$
we have that
\begin{enumerate}[(i)]
\item $A(x,u, \uu{w}, g_0(\uu{w},k),k+f_0(\uu{w},k))$; \label{po1}
\item $A(\vp(\uu{w}),u,\uu{w}',g'_0(\uu{w},k,\uu{w}',k'),k'+f'_0(\uu{w},k,\uu{w}',k'))$; \label{po2}
\item $A(x,u'(\uu{w},k,\uu{w}',k'),\uu{v},g_1(\uu{w},k,\uu{w}',k',\uu{v},h),h+f_1(\uu{w},k,\uu{w}',k',\uu{v},h))$;\\
$h \geq \iota(\uu{w},k,\uu{w}',k')$; \label{po3}
\item $A(x,u'(\uu{w},k,\uu{w}',k'),\uu{v},g_2(\uu{w},k,\uu{w}',k',\uu{v},l),l+f_2(\uu{w},k,\uu{w}',k',\uu{v},l))$; \label{po4}
\item $A(x,u,\uu{w},g_2(\uu{w},k,\uu{w}',k',\uu{v},l),l+f_2(\uu{w},k,\uu{w}',k',\uu{v},l))$; \label{po5}
\item $A(\vp(\uu{w}),u'(\uu{w},k,\uu{w}',k'),\uu{v}',g'_1(\uu{w},k,\uu{w}',k',\uu{v}',h'),h'+f'_1(\uu{w},k,\uu{w}',k',\uu{v}',h'))$;\\
$h' \geq \iota(\uu{w},k,\uu{w}',k')$; \label{po6}
\item $A(\vp(\uu{w}),u'(\uu{w},k,\uu{w}',k'),\uu{v}',g'_2(\uu{w},k,\uu{w}',k',\uu{v}',l'),l'+f'_2(\uu{w},k,\uu{w}',k',\uu{v}',l'))$; \label{po7}
\item $A(\vp(\uu{w}),u,\uu{w}',g'_2(\uu{w},k,\uu{w}',k',\uu{v}',l'),l'+f'_2(\uu{w},k,\uu{w}',k',\uu{v}',l'))$. \label{po8}
\end{enumerate}
Take notice that:
\begin{enumerate}[1.]
\item By the discussion at the beginning of this subsection, we already have such a $\Psi$, but its form is not relevant for this proposition.
\item The exact form of the $\Phi$ will be given over the course of the proof.
\end{enumerate}
\end{proposition}

\begin{proof}
We shall first derive a purely qualitative version of the above. Namely, let $u$, $u'$, $\uu{g}$, $\uu{g}'$, $\uu{f}$, $\uu{f}'$, $\iota$, $\vp$ be given. We will show that there exist $\uu{w}$, $k$, $\uu{w}'$, $k'$, $\uu{v}$, $\uu{v}'$, $l$, $l'$, $h$, $h'$ such that \eqref{po1}-\eqref{po8} hold. It will then follow, by the functional interpretation, that these objects can be explicitly constructed. The first step will be to prove the ``non-metastable'' version of our hypothesis, which we do in the following claim.\\[1mm]

\noindent {\bf Claim.} For all $(y_n) \subseteq C$ and $\delta>0$, there are $\uu{w}$ and $k$ such that for all $\uu{z}$ and $m$,
$$A(y,\delta,\uu{w},\uu{z},k+m).$$
\noindent {\bf Proof of claim:} Suppose the opposite, so there are $(y_n) \subseteq C$ and $\delta>0$ such that for all $\uu{w}$ and $k$ there are $\uu{z}$ and $m$, such that it is not the case that
$$A(y,\delta,\uu{w},\uu{z},k+m).$$
Put, for any $\uu{w}$ and $k$, $(g,f)(\uu{w},k)$ to be such a $\uu{z}$ and $m$. Then, for all $\uu{w}$ and $k$,
$$\neg A(y,\delta,\uu{w},g(\uu{w},k),k+f(\uu{w},k)).$$
If we now put $(\uu{w},k):=\Psi(y,\delta,g,f)$, we contradict our hypothesis.
\hfill $\blacksquare$\\[2mm]

If we apply the Claim to $(x,u)$, we get $\uu{w}$ and $k$ such that for all $\uu{z}$ and $m$,
\begin{equation}\label{c-e1}
A(x,u,\uu{w},\uu{z},k+m),
\end{equation}
from which we get \eqref{po1}. Apply then the Claim to $(\vp(\uu{w}),u)$ to get $\uu{w}'$ and $k'$ such that for all $\uu{z}$ and $m$,
$$A(\vp(\uu{w}),u,\uu{w}',\uu{z},k'+m),$$
from which we get \eqref{po2}. Now apply the Claim to $(x,u'(\uu{w},k,\uu{w}',k'))$ to get $\uu{v}$ and $h_0$ such that for all $\uu{z}$ and $m$,
\begin{equation}\label{c-e2}
A(x,u'(\uu{w},k,\uu{w}',k'),\uu{v},\uu{z},h_0+m).
\end{equation}
Put $h:=h_0+\iota(\uu{w},k,\uu{w}',k')$. Then we have that for all $\uu{z}$ and $m$,
$$A(x,u'(\uu{w},k,\uu{w}',k'),\uu{v},\uu{z},h+m),$$
and so we get \eqref{po3}. Let $l:=k+h_0$. From \eqref{c-e2}, we have that for all $\uu{z}$ and $m$,
$$A(x,u'(\uu{w},k,\uu{w}',k'),\uu{v},\uu{z},l+m),$$
from which we get \eqref{po4}. Similarly, from \eqref{c-e1}, we have that for all $\uu{z}$ and $m$,
$$A(x,u,\uu{w},\uu{z},l+m),$$
from which we get \eqref{po5}. Afterwards, $\uu{v}'$, $l'$ and $h'$ -- and thus \eqref{po6}, \eqref{po7} and \eqref{po8} -- are obtained in a similar manner.\hfill $\blacksquare$\\[2mm]

Now we proceed to the construction of $\Phi$. Since the above proof used only pure logic and the basic properties of the operation of addition, it follows by the soundness theorem of the functional interpretation that $\Phi$ can be constructed out of just $\lambda$-terms, $+$ and case distinction. When we shall majorize $\Phi$ to get our final bound, the case distinction will disappear, being replaced by the maximum. This is why we do not need to solve the case distinction further (which we could, by using suitable rational approximations, in order for the $\Phi$ to be fully constructive).

For conceptual clarity, we shall split the proof analysis into two distinct parts, the purely logical one and the ``mathematical'' one (which uses addition). Define the following functionals:
$$\Xi(\xi,\xi')(g,f)(\uu{x}):=(g(x_{1-5},x_6+\xi(\uu{x})),\xi'(\uu{x})+f(x_{1-5},x_6+\xi(\uu{x})))$$
$$\xi_0(\uu{x}):=x_6\qquad \xi_1(\uu{x}):=\iota(x_{1-4})\qquad \xi_2(\uu{x}):=x_2\qquad \xi'_2(\uu{x}):=x_4$$
$$(\tilde{g}_1,\tilde{f}_1):=\Xi(\xi_1,\xi_1)(g_1,f_1)\qquad (\tilde{g}_2,\tilde{f}_2):=\Xi(\xi_2,\xi_2)(g_2,f_2)\qquad (\tilde{g}_3,\tilde{f}_3):=\Xi(\xi_2,\xi_0)(g_2,f_2)$$
$$(\tilde{g}'_1,\tilde{f}'_1):=\Xi(\xi_1,\xi_1)(g'_1,f'_1)\qquad (\tilde{g}'_2,\tilde{f}'_2):=\Xi(\xi'_2,\xi'_2)(g'_2,f'_2)\qquad (\tilde{g}'_3,\tilde{f}'_3):=\Xi(\xi'_2,\xi_0)(g'_2,f'_2)$$
\mbox{}

\noindent {\bf Claim.} There exist $\uu{w}$, $k$, $\uu{w}'$, $k'$, $\uu{v}$, $\tilde{h}$, $\uu{v}'$, $\tilde{h}'$ such that 
\begin{enumerate}[(i')]
\item $A(x,u, \uu{w}, g_0(\uu{w},k),k+f_0(\uu{w},k))$; \label{po1b}
\item $A(\vp(\uu{w}),u,\uu{w}',g'_0(\uu{w},k,\uu{w}',k'),k'+f'_0(\uu{w},k,\uu{w}',k'))$; \label{po2b}
\item $A(x,u'(\uu{w},k,\uu{w}',k'),\uu{v},\tilde{g}_1(\uu{w},k,\uu{w}',k',\uu{v},\tilde{h}),\tilde{h}+\tilde{f}_1(\uu{w},k,\uu{w}',k',\uu{v},\tilde{h}))$;\label{po3b}
\item $A(x,u'(\uu{w},k,\uu{w}',k'),\uu{v},\tilde{g}_2(\uu{w},k,\uu{w}',k',\uu{v},\tilde{h}),\tilde{h}+\tilde{f}_2(\uu{w},k,\uu{w}',k',\uu{v},\tilde{h}))$; \label{po4b}
\item $A(x,u,\uu{w},\tilde{g}_3(\uu{w},k,\uu{w}',k',\uu{v},\tilde{h}),k+\tilde{f}_3(\uu{w},k,\uu{w}',k',\uu{v},\tilde{h}))$; \label{po5b}
\item $A(\vp(\uu{w}),u'(\uu{w},k,\uu{w}',k'),\uu{v}',\tilde{g}'_1(\uu{w},k,\uu{w}',k',\uu{v}',\tilde{h}'),\tilde{h}'+\tilde{f}'_1(\uu{w},k,\uu{w}',k',\uu{v}',\tilde{h}'))$;\label{po6b}
\item $A(\vp(\uu{w}),u'(\uu{w},k,\uu{w}',k'),\uu{v}',\tilde{g}'_2(\uu{w},k,\uu{w}',k',\uu{v}',\tilde{h}'),\tilde{h}'+\tilde{f}'_2(\uu{w},k,\uu{w}',k',\uu{v}',\tilde{h}'))$; \label{po7b}
\item $A(\vp(\uu{w}),u,\uu{w}',\tilde{g}'_3(\uu{w},k,\uu{w}',k',\uu{v}',\tilde{h}'),k'+\tilde{f}'_3(\uu{w},k,\uu{w}',k',\uu{v}',\tilde{h}'))$. \label{po8b}
\end{enumerate}

\noindent {\bf Proof of claim:} Define $\uu{w},k,\uu{w}',k',\uu{v},\tilde{h},\uu{v}',\tilde{h}'$ in the following way:

\begin{align*}
(g_v,f_v)(\uu{q}) &:=
  \begin{cases} 
       (\tilde{g}_1,\tilde{f}_1)(\uu{q}),  \textup{ if it is not the case that } A(x,u'(q_{1-4}), q_5, \tilde{g}_1(\uu{q}), q_6 + \tilde{f}_1(\uu{q})), \\
       (\tilde{g}_2,\tilde{f}_2)(\uu{q}),  \hfill\textup{ otherwise.} \\
  \end{cases}\\ 
a_v(\uu{r}) &:=\Psi(x,u'(\uu{r}),\lambda\uu{s}.(g_v,f_v)(\uu{r},\uu{s})).\\
(g_{v'},f_{v'})(\uu{q}) &:=
  \begin{cases} 
       (\tilde{g}'_1,\tilde{f}'_1)(\uu{q}),  \textup{ if it is not the case that } A(\vp(q_1),u'(q_{1-4}), q_5, \tilde{g}'_1(\uu{q}), q_6 + \tilde{f}'_1(\uu{q})), \\
       (\tilde{g}'_2,\tilde{f}'_2)(\uu{q}),  \hfill\textup{ otherwise.} \\
  \end{cases}\\ 
a_{v'}(\uu{r}) &:=\Psi(\vp(r_1),u'(\uu{r}),\lambda\uu{s}.(g_{v'},f_{v'})(\uu{r},\uu{s})).\\
(g_{w'},f_{w'})(\uu{q}) &:=
  \begin{cases} 
       (g'_0,f'_0)(\uu{q}),  \textup{ if it is not the case that } A(\vp(q_1),u, q_3, g'_0(\uu{q}), q_4 + f'_0(\uu{q})), \\
       (\tilde{g}'_3,\tilde{f}'_3)(\uu{q},a_{v'}(\uu{q})),  \hfill\textup{ otherwise.} \\
  \end{cases}\\ 
a_{w'}(\uu{r}) &:=\Psi(\vp(r_1),u,\lambda\uu{s}.(g_{w'},f_{w'})(\uu{r},\uu{s})).\\
(g_{w},f_{w})(\uu{q}) &:=
  \begin{cases} 
       (g_0,f_0)(\uu{q}),  \textup{ if it is not the case that } A(x,u, q_1, g_0(\uu{q}), q_2 + f_0(\uu{q})), \\
       (\tilde{g}_3,\tilde{f}_3)(\uu{q},a_{w'}(\uu{q}),a_v(\uu{q},a_{w'}(\uu{q}))),  \hfill\textup{ otherwise.} \\
  \end{cases}\\ 
(\uu{w},k) &:= \Psi(x,u,(g_w,f_w)).\\
(\uu{w}',k') &:= a_{w'}(\uu{w},k).\\
(\uu{v},\tilde{h}) &:= a_v(\uu{w},k,\uu{w}',k').\\
(\uu{v}',\tilde{h}') &:= a_{v'}(\uu{w},k,\uu{w}',k').
\end{align*}

We first apply the hypothesis on $\Psi$ to $(x,u,(g_w,f_w))$. Since $(\uu{w},k) = \Psi(x,u,(g_w,f_w))$, we have that
\begin{equation}\label{aa1}
A(x,u,\uu{w},g_w(\uu{w},k),k+f_w(\uu{w},k)).
\end{equation}
Suppose that it is not the case that
\begin{equation}\label{aa2}
A(x,u,\uu{w},g_0(\uu{w},k),k+f_0(\uu{w},k)).
\end{equation}
Then, by the definition of $(g_w,f_w)$, we have that $(g_w,f_w)(\uu{w},k)=(g_0,f_0)(\uu{w},k)$, so, by \eqref{aa1}, we have that \eqref{aa2} holds, which is a contradiction. Therefore \eqref{aa2} holds, so $(g_w,f_w)(\uu{w},k)=(\tilde{g}_3,\tilde{f}_3)(\uu{w},k,\uu{w}',k',\uu{v},\tilde{h})$ and
\begin{equation}\label{aa3}
A(x,u,\uu{w},\tilde{g}_3(\uu{w},k,\uu{w}',k',\uu{v},\tilde{h}),k+\tilde{f}_3(\uu{w},k,\uu{w}',k',\uu{v},\tilde{h})).
\end{equation}
We have thus proven the first and fifth item on our list. The other three pairs of items are proven in the same way, by applying the hypothesis on $\Psi$ to 
$$(\vp(\uu{w}),u,\lambda\uu{s}.(g_{w'},f_{w'})(\uu{w},k,\uu{s})),$$
$$(\vp(\uu{w}),u'(\uu{w},k,\uu{w}',k'),\lambda\uu{s}.(g_{v'},f_{v'})(\uu{w},k,\uu{w}',k',\uu{s})),$$
and
$$(x,u'(\uu{w},k,\uu{w}',k'),\lambda\uu{s}.(g_v,f_v)(\uu{w},k,\uu{w}',k',\uu{s})),$$
successively.
\hfill $\blacksquare$\\[2mm]

To finish the proof, we need only to use the $\uu{w}$, $k$, $\uu{w}'$, $k'$, $\uu{v}$, $\uu{v}'$ already obtained in the claim and then to put

$$h:=\tilde{h} + \iota(\uu{w},k,\uu{w}',k'),\quad  h':=\tilde{h}' + \iota(\uu{w},k,\uu{w}',k'),$$
$$l:=k+\tilde{h},\quad l':=k'+\tilde{h}'.$$

Then the items \eqref{po1}-\eqref{po8} follow from the corresponding ones in the claim by a simple verification using the above definitions and the earlier ones of $\uu{\tilde{g}}$, $\uu{\tilde{g}}'$, $\uu{\tilde{f}}$ and $\uu{\tilde{f}}'$.
\end{proof}
\begin{remark} The case distinctions made in defining the various functions 
in the proof of the claim above (and also in some proofs below) 
serve to achieve that the value produced 
simultaneously satisfies two requirements. This is reminiscent of the 
treatment of the logical contraction axiom $A\to A\wedge A$ in G\"odel's 
functional (`Dialectica') 
interpretation {\rm \cite{Goe58}}. In the end, when computing the 
bound we are interested in by a process of majorization (monotone functional 
interpretation, see {\rm \cite{Koh08}}), we can always just take the maximum of the two values 
and so the case distinctions are not needed to be computed but serve to 
justify why the bound is correct. Alternatively, the correctness of 
taking the maximum can also be argued for by using the so-called bounded 
functional interpretation {\rm \cite{FerOl05}} which globally changes the whole interpretation whereas we prefer our local verification as this does not 
require to actually spell out the general interpretation.
\end{remark}
\begin{lemma}\label{abstract-lemma-two}
Let $b \in \N^*$, $X$ be a Banach space, $C \subseteq X$ be a set of diameter at most $b$ and $(x_n) \subseteq C$. Assume that there are suitable $\delta$, $\delta'$, $\widetilde{\delta}$, $\uu{w}$, $\uu{w}'$, $\uu{q}$, $m$, $n$, $z$, $N$, $T$, $P$, $k$, $U$, $M$ such that:
\begin{enumerate}[(i)]
\item $A(x,\delta,\uu{w},\uu{q},m+n)$; $A(x,\delta',\uu{w}',\uu{q},m+n)$;
\item $0 < \delta \leq \widetilde{\delta}$; $0 < \delta' \leq \widetilde{\delta}$;
\item $z=q_2$;
\item $N$, $T$ and $P$ are the functionals defined in Proposition~\ref{exlimsup-nd}, customized by instantiating their free parameters with $b \mapsto b^2$, $k \mapsto \left\lceil\frac{4}{\min(\delta,\delta')}\right\rceil$ and $(a_n) \mapsto (\|x_n-z\|^2)$;
\item $k = \left\lceil\frac{4}{\min(\delta,\delta')}\right\rceil$;
\item $(U,M)$, for an arbitrary argument $\uu{v}$, has the following value: if it is not the case that
$$0 \leq v_3 \leq b^2 \cdot (k+1),\ \|x_{m+v_1(m)}-z\|^2 \geq \frac{v_3}{k+1} - \frac1{k+1}  \text{ and } \|x_{v_2+v_1(m)}-z\|^2 \leq \frac{v_3}{k+1} + \frac1{k+1},$$
then $(v_1(m),m)$, else if it is not the case that
$$0 \leq v_3 \leq b^2 \cdot (k+1),\ \|x_{w_1(v_3,z,v_1,v_2) + v_1(w_1(v_3,z,v_1,v_2))}-z\|^2 \geq \frac{v_3}{k+1} - \frac1{k+1}$$
and
$$\|x_{v_2+w_2(v_3,z,v_1,v_2)}-z\|^2 \leq \frac{v_3}{k+1} + \frac1{k+1},$$
then $(w_2,w_1)(v_3,z,v_1,v_2)$, else $(w'_2,w'_1)(v_3,z,v_1,v_2)$;
\item $(q_1,q_3,q_4)=(P,N,T)(U,M)$;
\item $q_5=0$;
\item $n=NUM(m)$.
\end{enumerate}
Then we have that
$$\|x_{m+n}-w_4\|^2 - \|x_{m+n}-z\|^2 \leq \widetilde{\delta} \text{ and } \|x_{m+n}-w'_4\|^2 - \|x_{m+n}-z\|^2 \leq \widetilde{\delta}.$$
\end{lemma}

\begin{proof}
By the definition of $A$, we have that
$$0 \leq w_3 \leq b^2 \cdot (k+1),\ \|x_{w_5(0)}-w_4\|^2 \geq \frac{w_3}{k+1} - \frac{\delta}4  \text{ and } \|x_{m+NUM(m)}-w_4\|^2 \leq \frac{w_3}{k+1} + \frac{\delta}4$$
and that if
$$0 \leq PUM \leq b^2 \cdot (k+1),\ \|x_{w_1(PUM,z,NUM,TUM)+(NUM)(w_1(PUM,z,NUM,TUM))}-z\|^2 \geq \frac{PUM}{k+1} - \frac{\delta}4$$
and
$$\|x_{TUM+w_2(PUM,z,NUM,TUM)}-z\|^2 \leq \frac{PUM}{k+1} + \frac{\delta}4$$
then
$$\frac{w_3}{k+1} \leq \frac{PUM}{k+1} + \frac{\delta}2.$$

The second instance of $A$ shows that
$$0 \leq w'_3 \leq b^2 \cdot (k+1),\ \|x_{w'_5(0)}-w'_4\|^2 \geq \frac{w'_3}{k+1} - \frac{\delta'}4  \text{ and } \|x_{m+NUM(m)}-w'_4\|^2 \leq \frac{w'_3}{k+1} + \frac{\delta'}4$$
and that if
$$0 \leq PUM \leq b^2 \cdot (k+1),\ \|x_{w'_1(PUM,z,NUM,TUM)+(NUM)(w'_1(PUM,z,NUM,TUM))}-z\|^2 \geq \frac{PUM}{k+1} - \frac{\delta'}4$$
and
$$\|x_{TUM+w'_2(PUM,z,NUM,TUM)}-z\|^2 \leq \frac{PUM}{k+1} + \frac{\delta'}4$$
then
$$\frac{w'_3}{k+1} \leq \frac{PUM}{k+1} + \frac{\delta'}2.$$

By Proposition~\ref{exlimsup-nd} and the condition on $(N,T,P)$, we get that
\begin{equation}\label{p-1}
\left.
\begin{gathered}
0 \leq PUM \leq b^2 \cdot (k+1),\\
\|x_{M(NUM,TUM,PUM)+(NUM)(M(NUM,TUM,PUM))}-z\|^2 \geq \frac{PUM}{k+1} - \frac1{k+1}\\
\text{and}\\
\|x_{TUM+U(NUM,TUM,PUM)}-z\|^2 \leq \frac{PUM}{k+1} + \frac1{k+1}.
\end{gathered}
\right\}
\end{equation}

By the condition on $(U,M)$, if it is not the case that
\begin{equation}\label{c-1-1}
\left.
\begin{gathered}
0 \leq PUM \leq b^2 \cdot (k+1),\ \|x_{m+NUM(m)}-z\|^2 \geq \frac{PUM}{k+1} - \frac1{k+1}\\
\text{and}\\
\|x_{TUM+NUM(m)}-z\|^2 \leq \frac{PUM}{k+1} + \frac1{k+1},
\end{gathered}
\right\}
\end{equation}
then $(U,M)(NUM,TUM,PUM)=(NUM(m),m)$. By \eqref{p-1}, it follows that \eqref{c-1-1} holds, contradicting our assumption. Therefore, indeed, \eqref{c-1-1} holds. Suppose now that it is not the case that
\begin{equation}\label{c-2-1}
\left.
\begin{gathered}
0 \leq PUM \leq b^2 \cdot (k+1),\\
\|x_{w_1(PUM,z,NUM,TUM)+(NUM)(w_1(PUM,z,NUM,TUM))}-z\|^2 \geq \frac{PUM}{k+1} - \frac1{k+1}\\
\text{and}\\
\|x_{TUM+w_2(PUM,z,NUM,TUM)}-z\|^2 \leq \frac{PUM}{k+1} + \frac1{k+1}.
\end{gathered}
\right\}
\end{equation}
Then $(U,M)(NUM,TUM,PUM)=(w_2,w_1)(PUM,z,NUM,TUM)$. By \eqref{p-1}, it follows that \eqref{c-2-1} holds, contradicting our assumption. Therefore, indeed, \eqref{c-2-1} holds. In particular, since $\frac1{k+1} \leq \frac{\delta}4$,
$$0 \leq PUM \leq b^2 \cdot (k+1),\ \|x_{w_1(PUM,z,NUM,TUM)+(NUM)(w_1(PUM,z,NUM,TUM))}-z\|^2 \geq \frac{PUM}{k+1} - \frac{\delta}4$$
and
$$\|x_{TUM+w_2(PUM,z,NUM,TUM)}-z\|^2 \leq \frac{PUM}{k+1} + \frac{\delta}4.$$
Thus, $(U,M)(NUM,TUM,PUM)=(w'_2,w'_1)(PUM,z,NUM,TUM)$. Yet again, by \eqref{p-1}, it follows that
$$0 \leq PUM \leq b^2 \cdot (k+1),\ \|x_{w'_1(PUM,z,NUM,TUM)+(NUM)(w'_1(PUM,z,NUM,TUM))}-z\|^2 \geq \frac{PUM}{k+1} - \frac1{k+1}$$
and
$$\|x_{TUM+w'_2(PUM,z,NUM,TUM)}-z\|^2 \leq \frac{PUM}{k+1} + \frac1{k+1},$$
so, since $\frac1{k+1} \leq \frac{\delta'}4$,
$$0 \leq PUM \leq b^2 \cdot (k+1),\ \|x_{w'_1(PUM,z,NUM,TUM)+(NUM)(w'_1(PUM,z,NUM,TUM))}-z\|^2 \geq \frac{PUM}{k+1} - \frac{\delta'}4$$
and
$$\|x_{TUM+w'_2(PUM,z,NUM,TUM)}-z\|^2 \leq \frac{PUM}{k+1} + \frac{\delta'}4.$$

Therefore, we have that
$$\frac{w_3}{k+1} \leq \frac{PUM}{k+1} + \frac{\delta}2$$
and
$$\frac{w'_3}{k+1} \leq \frac{PUM}{k+1} + \frac{\delta'}2.$$

We may now compute:
\begin{align*}
\|x_{m+n}-w_4\|^2 &= \|x_{m+NUM(m)}-w_4\|^2 \\
 &\leq \frac{w_3}{k+1} + \frac{\delta}4 \\
&\leq \frac{PUM}{k+1} + \frac{3\delta}4 \\
&\leq \|x_{m+NUM(m)}-z\|^2 + \frac{3\delta}4 + \frac1{k+1} \\
&\leq \|x_{m+NUM(m)}-z\|^2 + \widetilde{\delta},\\
&= \|x_{m+n}-z\|^2 + \widetilde{\delta},
\end{align*}
and similarly we obtain
$$\|x_{m+n}-w'_4\|^2 \leq \|x_{m+n}-z\|^2 + \widetilde{\delta},$$
so we are done.
\end{proof}

The following corollary is simply the instantiation of Lemma~\ref{abstract-lemma-two} above for $\delta':=\delta$, $\widetilde{\delta}:=\delta$ and $\uu{w}':=\uu{w}$.

\begin{corollary}\label{abstract-lemma-one}
Let $b \in \N^*$, $X$ be a Banach space, $C \subseteq X$ be a set of diameter at most $b$ and $(x_n) \subseteq C$. Assume that there are suitable $\delta$, $\uu{w}$, $\uu{q}$, $m$, $n$, $z$, $N$, $T$, $P$, $k$, $U$, $M$ such that:
\begin{enumerate}[(i)]
\item $A(x,\delta,\uu{w},\uu{q},m+n)$;
\item $\delta > 0$;
\item $z=q_2$;
\item $N$, $T$ and $P$ are the functionals defined in Proposition~\ref{exlimsup-nd}, customized by instantiating their free parameters with $b \mapsto b^2$, $k \mapsto \left\lceil\frac{4}{\delta}\right\rceil$ and $(a_n) \mapsto (\|x_n-z\|^2)$;
\item $k = \left\lceil\frac{4}{\delta}\right\rceil$;
\item $(U,M)$, for an arbitrary argument $\uu{v}$, has the following value: if it is not the case that
$$0 \leq v_3 \leq b^2 \cdot (k+1),\ \|x_{m+v_1(m)}-z\|^2 \geq \frac{v_3}{k+1} - \frac1{k+1}  \text{ and } \|x_{v_2+v_1(m)}-z\|^2 \leq \frac{v_3}{k+1} + \frac1{k+1},$$
then $(v_1(m),m)$, else $(w_2,w_1)(v_3,z,v_1,v_2)$;
\item $(q_1,q_3,q_4)=(P,N,T)(U,M)$;
\item $q_5=0$;
\item $n=NUM(m)$.
\end{enumerate}
Then we have that
$$\|x_{m+n}-w_4\|^2 - \|x_{m+n}-z\|^2 \leq \delta.$$
\end{corollary}

\subsubsection{The extraction of the quantities in Claim II}\label{quantities}

We will now show how to prove Claim II in the second proof of Theorem~\ref{main-thm} (the one that uses approximate limsup's). We use the notations introduced before the statement of that claim.

We will now define $u, u', \uu{g}, \uu{g}', \uu{f}, \uu{f}', \iota, \vp$.

\begin{enumerate}[I.]

\item The definition of $u$ and $\vp$.

These quantities are defined analogously to the corresponding ones in (the proof of) Claim II. First put

$$u:= \min \left\{ \frac{2\nu_4(\eps)\delta(\eps)}3, \nu_2(\eps) \right\}$$

For the $\vp$, we just have to extract the point $w$ (i.e. the fourth component) from $\uu{w}$ and then form the sequence $(x^w_n)$ as before.

\item The definition of $g_0$ and $f_0$.

Consider some arbitrary $(\widetilde{U},\widetilde{N},p,y,L,m)$ for their arguments.

Let $z$ be $y + \delta(\eps)(x-y)$ and $N$, $T$ and $P$ be the functionals defined in Proposition~\ref{exlimsup-nd}, customized by instantiating their free parameters with $b \mapsto b^2$, $k \mapsto \left\lceil\frac{4}{u}\right\rceil$ and $(a_n) \mapsto (\|x_n-z\|^2)$. We continue to use in the following the notation $k := \left\lceil\frac{4}{u}\right\rceil$. We define $(U,M)$, for an arbitrary argument $\uu{v}$, in the following way: if it is not the case that
$$0 \leq v_3 \leq b^2 \cdot (k+1),\ \|x_{m+v_1(m)}-z\|^2 \geq \frac{v_3}{k+1} - \frac1{k+1}  \text{ and } \|x_{v_2+v_1(m)}-z\|^2 \leq \frac{v_3}{k+1} + \frac1{k+1}$$
then put their value as $(v_1(m),m)$, else put it as $(\widetilde{N},\widetilde{U})(v_3,z,v_1,v_2)$. Finally, put the value of $g_0$ to be $(PUM,z,NUM,TUM,0)$ and the value of $f_0$ to be $(NUM)(m)$.

\item The definition of $g'_0$ and $f'_0$.

Consider some arbitrary $(\widetilde{U},\widetilde{N},p,y,L,m,\widetilde{U}',\widetilde{N}',p',y',L',m')$ for their arguments.

Let $z$ be $y' + \delta(\eps)(x-y')$ and $N$, $T$ and $P$ be the functionals defined in Proposition~\ref{exlimsup-nd}, customized by instantiating their free parameters with $b \mapsto b^2$, $k \mapsto \left\lceil\frac{4}{u}\right\rceil$ and $(a_n) \mapsto (\|x^y_n-z\|^2)$. We continue to use in the following the notation $k := \left\lceil\frac{4}{u}\right\rceil$. We define $(U,M)$, for an arbitrary argument $\uu{v}$, in the following way: if it is not the case that
$$0 \leq v_3 \leq b^2 \cdot (k+1),\ \|x^y_{m'+v_1(m')}-z\|^2 \geq \frac{v_3}{k+1} - \frac1{k+1}  \text{ and } \|x^y_{v_2+v_1(m')}-z\|^2 \leq \frac{v_3}{k+1} + \frac1{k+1}$$
then put their value as $(v_1(m'),m')$, else put it as $(\widetilde{N}',\widetilde{U}')(v_3,z,v_1,v_2)$. Finally, put the value of $g'_0$ to be $(PUM,z,NUM,TUM,0)$ and the value of $f'_0$ to be $(NUM)(m')$.

\item The definition of $u'$ and $\iota$.

Consider some arbitrary $(\uu{w},k,\uu{w}',k')$ for their arguments.

Put
$$\widetilde{k}:=k+f_0(\uu{w},k)\text{ and }\widetilde{k}':=k'+f_0'(\uu{w},k,\uu{w}',k').$$
Now put the value of $u'$ to be
$$\min \left\{ \frac{\nu_1(w_2,\widetilde{k},\widetilde{k}',\eps)}2, \nu_2(\eps) \right\}$$
and the one of $\iota$ to be
$$\alpha\left(\left\lceil\max\left\{\frac{2b}{\sqrt{\nu_1(w_2,\widetilde{k},\widetilde{k}',\eps)}},\frac{8b^2}{\nu_1(w_2,\widetilde{k},\widetilde{k}',\eps)}\right\}\right\rceil\right).$$

\item The definition of $g_1$ and $f_1$.

These are defined similarly to $(g_0,f_0)$ and $(g'_0,f'_0)$, with the caveat that we need access to $(\uu{w},k,\uu{w}',k')$ in order to work with the value $u'(\uu{w},k,\uu{w}',k')$ when defining the corresponding $N$, $T$ and $P$.

\item The definition of $g_2$ and $f_2$.

These will play a role in the application of Lemma~\ref{abstract-lemma-two}, so we step carefully through their definition.

Consider some arbitrary $(\widetilde{U},\widetilde{N},p,w,L,m,\widetilde{U}',\widetilde{N}',p',w',L',m',\widetilde{U}'',\widetilde{N}'',p'',v,L'',l)$ for their arguments.

Let $z$ be $\frac{v+w}2$ and $N$, $T$ and $P$ be the functionals defined in Proposition~\ref{exlimsup-nd}, customized by instantiating their free parameters with $b \mapsto b^2$, $k \mapsto \left\lceil\frac{4}{\min\{u,u'(\uu{w},k,\uu{w}',k')\}}\right\rceil$ and $(a_n) \mapsto (\|x_n-z\|^2)$. We continue to use in the following the notation $k := \left\lceil\frac{4}{\min\{u,u'(\uu{w},k,\uu{w}',k')\}}\right\rceil$. We define $(U,M)$, for an arbitrary argument $\uu{v}$, in the following way: if it is not the case that
$$0 \leq v_3 \leq b^2 \cdot (k+1),\ \|x_{m+v_1(m)}-z\|^2 \geq \frac{v_3}{k+1} - \frac1{k+1}  \text{ and } \|x_{v_2+v_1(m)}-z\|^2 \leq \frac{v_3}{k+1} + \frac1{k+1}$$
then $(v_1(m),m)$, else if it is not the case that
$$0 \leq v_3 \leq b^2 \cdot (k+1),\ \|x_{\widetilde{U}(v_3,z,v_1,v_2) + v_1(\widetilde{U}(v_3,z,v_1,v_2))}-z\|^2 \geq \frac{v_3}{k+1} - \frac1{k+1}$$
and
$$\|x_{v_2+\widetilde{N}(v_3,z,v_1,v_2)}-z\|^2 \leq \frac{v_3}{k+1} + \frac1{k+1},$$
then put their value as $(\widetilde{N},\widetilde{U})(v_3,z,v_1,v_2)$, else put it as $(\widetilde{N}'',\widetilde{U}'')(v_3,z,v_1,v_2)$. Finally, put the value of $g_2$ to be $(PUM,z,NUM,TUM,0)$ and the value of $f_2$ to be $(NUM)(m)$.

\item The definition of $g'_1$ and $f'_1$.

These are defined similarly to $(g_1,f_1)$.

\item The definition of $g'_2$ and $f'_2$.

These are defined similarly to $(g_2,f_2)$.
\end{enumerate}

Now that we have defined $u, u', \uu{g}, \uu{g}', \uu{f}, \uu{f}', \iota, \vp$, put
$$(\uu{w},k,\uu{w}',k',\uu{v},\uu{v}',l,l',h,h'):=\Phi(\Psi)(u, u', \uu{g}, \uu{g}', \uu{f}, \uu{f}', \iota, \vp)$$
and apply Proposition~\ref{lambda-term}.

Claim II then follows by applying Corollary~\ref{abstract-lemma-one} four times and Lemma~\ref{abstract-lemma-two} two times, and then performing some simple computations similar to the ones in the original proof of the claim. The relevant fact here is that the $N$ that witnesses the metastability property is equal to
$$k'+f_0'(\uu{w},k,\uu{w}',k').$$

\section{The rate of metastability}\label{sec:metastab}

When one has reached the end of the previous section, one can rightfully say that one is in the possession of a formula witnessing, for any $\eps$ and $g$, the rank corresponding to the metastable reformulation of the Cauchy property (depending on additional parameters of the problem). It is however not an effective formula and not uniform at all as it depends on all the data of the problem.
However by a process of majorization one easily obtains a bound (called a 
rate of metastability in the Introduction) which is both effective and highly 
uniform in the sense that it -- in addition to $\varepsilon$ and $g$ -- only depends 
on the norm bound $b$ and the moduli $\eta$, $\tau$, $\Theta$, $\alpha$, and $\gamma$ but not 
on $X$, $C$, $T$, or $(t_n)$ themselves.
In order to explain this approach, however, we need to first make a detour into the details of the calculus of functionals in which our final bound 
will be expressed.

The system $T$ of Hilbert-G\"odel, mentioned in the Introduction, is a system of functionals of finite types. Those finite types are defined inductively in the following way: there is a primitive type of natural numbers, and if we have two types $\rho$ and $\tau$, we have a type denoted by $\rho \to \tau$ of functions from elements of type $\rho$ to elements of type $\tau$. Therefore, there is e.g. a type of functions $f: \N^\N \to \N^{\left(\N^\N\right)}$. Product types are not built into the system, but they can be emulated by {\it currying}, i.e. the identification of $A^{B \times C}$ with $(A^B)^C$.  The functionals themselves are given by terms in this system, which are built up inductively by repeated application of variables and of constants for zero and successor, for basic combinatory operations, and lastly for recursion over natural numbers.

The crucial notion that we will make use of in the following is the one of {\it majorization}, introduced by Howard \cite{How73}. Majorization is a family of binary relations, i.e. on elements of each type $\rho$ one has a relation $\succeq_\rho$. It is defined inductively and, moreover, hereditarily: for two natural numbers $n$ and $m$ one has $n \succeq_{\N} m$ iff $n \geq m$ and if $f$ and $g$ are of type $\rho \to \tau$ then $f \succeq_{\rho \to \tau} g$ iff for all $m$, $n$ of type $\rho$ with $m \succeq_\rho n$ one has $f(m) \succeq_\tau g(n)$. For example, to any $f: \N \to \N$, we associate the function $f^M: \N \to \N$, defined for any $n \in \N$, by
$$f^M(n):=\max_{0\leq i \leq n} f(i),$$
and it is immediate that $f^M \succeq_{\N^\N} f$ -- we say of $f^M$ that it {\it majorizes} or is a {\it majorant} for $f$. Not all elements of higher types admit a majorant, but all the constants of $T$ do, and by heredity this extends to all terms (containing only variable of types $\N,\N\to\N$) of $T$. 
As an illustrating example, if $f$ is defined recursively by a schema like the following (suppressing the type information and ignoring the definitions of $a$ and $g$):
\begin{align*}
f(0)&:=a\\
f(n+1)&:=g(n,f(n))
\end{align*}
then if $a^*$ and $g^*$ are majorants for $a$ and $g$, respectively, it is easy to check that the function $f^*$ defined by
\begin{align*}
f^*(0)&:=a^*\\
f^*(n+1)&:=\max(a^*,g^*(n,f^*(n)))
\end{align*}
majorizes $f$, where $\max$ for functionals is defined pointwise.

There is one further issue we need to take care of. In order to formalize arguments involving e.g. Banach spaces, one needs to extend (as was first done in \cite{Koh05}) the type system with a new primitive type $X$ corresponding to elements of such a space. To any such extended (`abstract') type $\rho$ one can then associate an ordinary type $\widehat{\rho}$ simply by replacing all the $X$'s with $\N$'s. Majorization is then defined in \cite{GerKoh08} on each abstract type $\rho$ as a binary relation between elements of type $\widehat{\rho}$ and those of type $\rho$, as follows: first we have that for any $x \in X$ and $n \in \N$, $n \succeq_{X} x$ iff $n \geq \|x\|$ and then one continues in the same hereditary manner as on the ordinary types.

As a consequence of all this, if one would majorize all the functionals that play a role into the definition of the witness obtained earlier, one would get a chance at finding a purely numerical (and thus uniform) rate of metastability in the sense defined above, after all the case distinctions are removed in favour of taking the corresponding maximum. This is what we will now proceed to do in a stepwise fashion.

First, we majorize the functionals introduced in Proposition~\ref{exlimsup-nd}. Those have three hidden parameters: the sequence itself, which is only used directly when defining $P$, and so it completely disappears by majorization, the $k$, which we show here explicitly, and the upper bound, which we instantiate with $b^2$, since that is the greatest possible bound on the sequences for which the approximate limsup's are obtained. Since the $P$ is trivially majorized by $b^2 \cdot (k^*+1)$, we omit it, since we can replace it by this value in its further appearances. Note that only the $N^*$ and the $T^*$ will play a role in further developments.

\begin{align*}
W^*U^*0&:= \mathcal{O}. \\
W^*U^*(n+1)&:= \lambda y.U^*(W^*U^*n,y,n). \\
J^*U^*M^*k^*0 &:= 0. \\
J^*U^*M^*k^*(n+1) &:= M^*(W^*U^*(b^2\cdot (k^*+1) ), J^*U^*M^*k^*n^*, b^2\cdot (k^*+1) ). \\
N^*U^*M^*k^* &:= W^*U^*(b^2 \cdot (k^*+1)). \\
T^*U^*M^*k^* &:= J^*U^*M^*k^*(b^2 \cdot (k^*+1)).
\end{align*}

We now do the same for the functionals in Proposition~\ref{claim-i}. Remark the added explicit parameter $l^*$. We specify that the variables $p^*$, $y^*$ and $m^*$ are of type $\N$, $L^*$ is in $\N^\N$ and the fifth and sixth components of $\Omega^*$ take values in $\N$ (this will be relevant for the calibration of the exact level of recursion that is needed to define these objects). Again, the only functional which we will use later is $\Psi^*$, the rest of them only serve to define it.

\begin{align*}
(M^*,U^*)(\Omega^*,0)&:=(\mathcal{O},\mathcal{O}).\\
(M^*,U^*)(\Omega^*,x+1)&:=\lambda p^*,y^*,L^*,m^*. (\Omega^*_{5,6}((M^*,U^*)(\Omega^*,x),p^*,y^*,L^*,m^*)).\\
I^*(l^*)&:=2b^2(l^*+1).\\
(\overline{M}^*,\overline{U}^*)(l^*,\Omega^*)&:=\lambda L^*,m^*,p^*. ((M^*,U^*)(\Omega^*,I^*(l^*)+1)(p^*,b,L^*,m^*)).\\
\WPsi^*(l^*,\Omega^*,0)&:= ((M^*,U^*)(\Omega^*,I^*(l^*)), 4b^2 \cdot (l^*+1),b,(N^*,T^*)((\overline{U}^*,\overline{M}^*)(l^*,\Omega^*))(4l^*+3)).\\
\WPsi^*(l^*,\Omega^*,x+1)&:=\max(\WPsi^*(l^*,\Omega^*,0),((M^*,U^*)(\Omega^*,I^*(l^*)),\Omega^*_{1-4}(\WPsi^*(l^*,\Omega^*,x)))).\\
\Psi^*(l^*,\Omega^*)&:=\WPsi^*(l^*,\Omega^*,I^*(l^*)),
\end{align*}
where $\max$ for tuples is understood componentwise. Now we begin the most intricate portion of the majorization procedure, namely the treatment of the functions defined in the final part of the previous section. We make some remarks in order to convince the reader of the plausibility of the solution given below. First of all, since $\vp$ yields a sequence which is bounded by $b$, it is trivially majorized, so we may omit it, like before with the $P$. In that same vein, we may omit some parameters if the majorant does not actually depend on them. For example, a majorant for $g'_0$ will now no longer depend on $\uu{w}$, and moreover it can be replaced by the same majorant as for $g_0$, provided that $\uu{w}^*$ and $k^*$ are replaced in applications by $\uu{w}'^*$ and $k'^*$. A case distinction may be replaced by a (pointwise) maximum (the verification is as immediate as for the recursion example given before), whereas when majorizing small real numbers $\delta > 0$ the maximum is, obviously, replaced by a minimum. For all undefined quantities below, see sub-section \ref{quantities}, as well as the notations introduced before the statement of Claim II in Section~\ref{section-approx}.

\begin{align*}
(N_1^*,T_1^*)U^*M^*&:= (N^*,T^*)U^*M^*\left\lceil \frac4{u} \right\rceil\\
(U_1^*,M_1^*)(\uu{w}^*,k^*,\uu{r}) &:= (\max(r_1(k^*),w_2^*(r_3,b,r_1,r_2)) , \max(k^*,w_1^*(r_3,b,r_1,r_2))).\\
(g_0^*,f_0^*)(\uu{w}^*,k^*)&:=\left( b^2\left(\left\lceil\frac4u \right\rceil +1 \right), b,(N_1^*,T_1^*)((U_1^*,M_1^*)(\uu{w}^*,k^*)),0,N_1^*((U_1^*,M_1^*)(\uu{w}^*,k^*))(k^*) \right).\\
\nu^*_1(m,n)&:=\frac12 \min_{c \leq \max(m,n+ g^M(n))} \psi_{b,\eta}(\widetilde{\theta}(\beta(c,\eps))).\qquad \widetilde{k}^*(\uu{w}^*,k^*) := k^* + f_0^*(\uu{w}^*,k^*).\\
u'^*(\uu{w}^*,k^*,\uu{w}'^*,k'^*) &:= \min \left( \frac12 \nu^*_1(\widetilde{k}^*(\uu{w}^*,k^*),\widetilde{k}^*(\uu{w}'^*,k'^*)), \nu_2(\eps) \right).\\
\iota^*(\uu{w}^*,k^*,\uu{w}'^*,k'^*) &:= \alpha^M\left(\left\lceil\max\left\{\frac{2b}{\sqrt{\nu^*_1(\widetilde{k}^*(\uu{w}^*,k^*),\widetilde{k}^*(\uu{w}'^*,k'^*))}},\frac{8b^2}{\nu^*_1(\widetilde{k}^*(\uu{w}^*,k^*),\widetilde{k}^*(\uu{w}'^*,k'^*))}\right\}\right\rceil\right).\\
\end{align*}\\[-1.6cm]
\begin{align*}
(N_2^*,T_2^*)(\uu{w}^*,k^*,\uu{w}'^*,k'^*)U^*M^*&:= (N^*,T^*)U^*M^*\left\lceil \frac4{u'^*(\uu{w}^*,k^*,\uu{w}'^*,k'^*)} \right\rceil.\\
(g_1^*,f_1^*)(\uu{w}^*,k^*,\uu{w}'^*,k'^*,\uu{v}^*,h^*)&:=\left( b^2\left(\left\lceil\frac4{u'^*(\uu{w}^*,k^*,\uu{w}'^*,k'^*)} \right\rceil +1 \right), b, \right.\\
& \qquad(N_2^*,T_2^*)(\uu{w}^*,k^*,\uu{w}'^*,k'^*)((U_1^*,M_1^*)(\uu{v}^*,h^*)),\\
&\qquad\left. \vphantom{\frac{3}{4}}0, N_1^*(\uu{w}^*,k^*,\uu{w}'^*,k'^*)((U_1^*,M_1^*)(\uu{v}^*,h^*))(h^*) \right).\\
(N_3^*,T_3^*)(\uu{w}^*,k^*,\uu{w}'^*,k'^*)U^*M^*&:= (N^*,T^*)U^*M^*\left\lceil \frac4{\min(u,u'^*(\uu{w}^*,k^*,\uu{w}'^*,k'^*))} \right\rceil.\\
(U_2^*,M_2^*)(\uu{w}^*,k^*,\uu{v}^*,l^*,\uu{r})& := (\max(r_1(k^*),w_2^*(r_3,b,r_1,r_2),v_2^*(r_3,b,r_1,r_2)),\\
&\qquad\max(k^*,w_1^*(r_3,b,r_1,r_2),v_1^*(r_3,b,r_1,r_2))).\\
(g_2^*,f_2^*)(\uu{w}^*,k^*,\uu{w}'^*,k'^*,\uu{v}^*,l^*)&:=\left( b^2\left(\left\lceil\frac4{\min(u,u'^*(\uu{w}^*,k^*,\uu{w}'^*,k'^*))} \right\rceil +1 \right), b, \right.\\
& \qquad(N_3^*,T_3^*)(\uu{w}^*,k^*,\uu{w}'^*,k'^*)((U_2^*,M_2^*)(\uu{w}^*,k^*,\uu{v}^*,l^*)),0,\\
&\qquad\left. \vphantom{\frac{3}{4}} N_3^*(\uu{w}^*,k^*,\uu{w}'^*,k'^*)((U_2^*,M_2^*)(\uu{w}^*,k^*,\uu{v}^*,l^*))(l^*) \right).\\
\end{align*}

We now majorize the functionals appearing in the proof of Proposition~\ref{lambda-term}. First, we treat the arithmetical shuffling stage.

$$\Xi(\xi,\xi')(g,f)(\uu{x}):=(g(x_{1-5},x_6+\xi(\uu{x})),\xi'(\uu{x})+f(x_{1-5},x_6+\xi(\uu{x})))$$
$$\xi_0(\uu{x}):=x_6\qquad \xi^*_1(\uu{x}):=\iota^*(x_{1-4})\qquad \xi_2(\uu{x}):=x_2\qquad \xi'_2(\uu{x}):=x_4$$
$$(\tilde{g}_1^*,\tilde{f}_1^*):=\Xi(\xi^*_1,\xi^*_1)(g_1^*,f_1^*)\qquad (\tilde{g}^*_2,\tilde{f}^*_2):=\Xi(\xi_2,\xi_2)(g^*_2,f^*_2)\qquad (\tilde{g}^*_3,\tilde{f}^*_3):=\Xi(\xi_2,\xi_0)(g^*_2,f^*_2)$$
$$(\tilde{g}'^*_2,\tilde{f}'^*_2):=\Xi(\xi'_2,\xi'_2)(g^*_2,f^*_2)\qquad (\tilde{g}'^*_3,\tilde{f}'^*_3):=\Xi(\xi'_2,\xi_0)(g^*_2,f^*_2)$$

Finally, we may treat the purely logical stage, where taking the maximum replaces the case distinctions. What we need to take care of is that the instance of $(g^*_0,f^*_0)$ majorizing $(g'_0,f'_0)$ is applied to its proper arguments, namely (here) $\uu{q}^*_{3-4}$.

\begin{align*}
(g^*_v,f^*_v)(\uu{q}^*) &:= \max((\tilde{g}^*_1,\tilde{f}^*_1)(\uu{q}^*), (\tilde{g}^*_2,\tilde{f}^*_2)(\uu{q}^*)).\\
a_v^*(\uu{r}^*) &:= \Psi^*\left( \left\lceil \frac{1}{u'^*(\uu{r}^*)} \right\rceil, \lambda\uu{s}.(g^*_v,f^*_v)(\uu{r}^*,\uu{s})\right).\\
(g^*_{v'},f^*_{v'})(\uu{q}^*) &:= \max((\tilde{g}^*_1,\tilde{f}^*_1)(\uu{q}^*), (\tilde{g}'^*_2,\tilde{f}'^*_2)(\uu{q}^*)).\\
a_{v'}^*(\uu{r}^*) &:= \Psi^*\left( \left\lceil \frac{1}{u'^*(\uu{r}^*)} \right\rceil, \lambda\uu{s}.(g^*_{v'},f^*_{v'})(\uu{r}^*,\uu{s})\right).\\
(g^*_{w'},f^*_{w'})(\uu{q}^*) &:= \max((g^*_0,f^*_0)(\uu{q}^*_{3-4}), (\tilde{g}_3'^*,\tilde{f}_3'^*)(\uu{q}^*,a_{v'}^*(\uu{q}^*))).\\
a_{w'}^*(\uu{r}^*) &:= \Psi^*\left( \left\lceil \frac{1}{u} \right\rceil, \lambda\uu{s}.(g^*_{w'},f^*_{w'})(\uu{r}^*,\uu{s})\right).\\
(g^*_{w},f^*_{w})(\uu{q}^*) &:= \max((g^*_0,f^*_0)(\uu{q}^*), (\tilde{g}_3^*,\tilde{f}_3^*)(\uu{q}^*,a^*_{w'}(\uu{q}^*),a^*_v(\uu{q}^*,a^*_{w'}(\uu{q}^*)))).\\
(\uu{w}^*,k^*) &:= \Psi^*\left( \left\lceil \frac{1}{u} \right\rceil, (g_w^*,f_w^*)\right).\\
(\uu{w}'^*,k'^*) &:= a^*_{w'}(\uu{w}^*,k^*).
\end{align*}

As before, one obtains the final bound of
$$k'^*+f_0^*(\uu{w}'^*,k'^*),$$
which, taking care of the dependencies in the formula just produced, we may denote as
$$\Theta'_{b,\eta,\tau,\theta,\alpha,\gamma}(\eps,g).$$

This, however, is not a rate of metastability in the sense that the notion was defined in the Introduction, but it may be easily converted into one. We remark that (suppressing the indices) $\Theta'$ depends on the $g$ only via $g^M$ and, moreover, the only property of $g^M$ that is used is that it is a majorant for $g$. Therefore, for any $h$ such that for any $n$, $h(n) \leq g(n)$, since then $g^M$ is also a majorant for $h$, $\Theta'(\eps,g)$ is a bound on the (least) $N$ such that $\|x_N - x_{N+h(N)}\| \leq \eps$. The uniformity of the bound having been already taken care of, we may allow the $h$ to depend on the sequence itself, and so
$$h(n):=\argmax_{0 \leq i \leq g(n)} \|x_{n+i} - x_n\|$$
is a valid choice (note that this $h$ can be made constructive by using suitable rational approximations). We may then simply put
$$\Theta_{b,\eta,\tau,\theta,\alpha,\gamma}(\eps,g):=\Theta'_{b,\eta,\tau,\theta,\alpha,\gamma}\left(\frac\eps2,g\right)$$
to obtain our main result, which is expressed as follows.

\begin{theorem}\label{main-metastab}
Let $X$ be a Banach space which is uniformly convex with modulus $\eta$ and uniformly smooth with modulus $\tau$. Let $C \subseteq X$ a closed, convex, nonempty subset. Let $b \in \N^*$ be such that for all $y \in C$, $\|y\| \leq b$ and the diameter of $C$ is bounded by $b$. Let $T: C \to C$ be a pseudocontraction that is uniformly continuous with modulus $\theta$ and $x \in C$. For all $t \in (0,1)$ put $x_t$ to be the unique point in $C$ such that $x_t = tTx_t + (1-t)x$. Let $(t_n) \subseteq (0,1)$, $\alpha : \N \to \N$ and $\gamma : \N \to \N^*$ be such that:
\begin{itemize}
\item for all $n$ and all $m \geq \alpha(n)$, $t_m \geq 1 - \frac1{n+1}$;
\item for all $n$, $t_n \leq 1 - \frac1{\gamma(n)}$.
\end{itemize}
Denote, for all $n$, $x_n:=x_{t_n}$. Then, for all $\eps > 0$ and $g: \N \to \N$ there is an $N \leq \Theta_{b,\eta,\tau,\theta,\alpha,\gamma}(\eps,g)$ such that for all $m,n \in [N,N+g(N)]$,
$\|x_m-x_n\| \leq \eps.$
\end{theorem}

Thus, we have obtained a rate of metastability which is definable in the subsystem of $T$ containing at most type two recursion, which we denote by $T_2$. Note that in order for our bound to be properly said to be defined in that calculus, one must take care that only natural numbers and functionals thereof are used in the definition. The most prominent examples of this sort are that one cannot work with an $\eps > 0$ and must instead use a natural approximation $k$ standing for $\eps:=\frac1{k+1}$, and also that the moduli of convexity, smoothness and continuity must also operate with and return natural numbers having this interpretation. This is straightforward to arrange (see also the metatheorems in 
\cite{Koh08,KohLeu12} which use the respective moduli in this form).

A closer look at the bound shows that type two recursion is only used in the 
definition of $(M^*,U^*)$ needed in defining $\Psi^*$ because of the argument $L^*$ which is in turn used by $\Omega^*$ . The concrete instances of $\Omega^*$ to which $\Psi^*$ is applied in the final stage, however, do not depend on that parameter, as it may be gleamed from a very careful examination.  To see this, 
it is crucial to note 
that the functionals $(U^*_i,M^*_i)$ do not depend on the fifth components 
of $\underline{w}^*,\underline{v}^*$ which play the role of $L^*.$ (That these 
functionals depend neither on the third nor the fourth component of 
$\underline{w}^*,\underline{v}^*$ is not surprising since these can be 
easily majorized in terms of $b$ and $b/\varepsilon$ for the 
respective error $\varepsilon$ which corresponds to the definition of the 
first and second components of the $g^*_i$'s.) Therefore one may replace that recursion by a (simpler) type one recursion\footnote{On the other hand, in the applications of $\Psi$ that were used to obtain the actual realizer, the parameter $L$ played a nondisposable role in the case distinction, but one can also make the remark that $L$ cannot play another role because in the proof of Lemma~\ref{abstract-lemma-two}, the corresponding `$\geq$' statements within the $A$'s were never used.}.
Note also that the primitive recursion of $\WPsi^*$ actually only concerns the components 3-6 (the first two ones have constant values) which are of types $\N$, $\N$, $\N^{\N}$ and $\N$ -- so that this is a recursion of type 
$\N^{\N}$. Actually, using again that the $\Omega^*_{1-4}$'s to which 
this recursion is applied do not depend on the type $\N^{\N}$ component 
$\tilde{\Psi}^*_5$ of $\tilde{\Psi}^*$, one can see that in the case at hand 
it reduces to a recursion of type $\N$. We also remark, that in the situation 
at hand, the functional $M^*$ (and hence $\overline{M}^*$) actually is 
constantly $0$ since the respective $\Omega^*_5$ functionals, namely the fifth 
components of the $g^*_i$'s, are $0$. 

\begin{corollary}
The bound $\Theta_{b,\eta,\tau,\theta,\alpha,\gamma}(1/(k+1),g),$ providing a rate of metastability for the resolvents of continuous pseudocontractive operators in Banach spaces which are uniformly convex and uniformly smooth, is definable in $T_1$ as a functional in the parameters $b$, $\eta$, $\tau$, $\theta$, $\alpha$, $\gamma$, $k$, $g$.
\end{corollary}
\begin{remark}\label{complexity} 
A detailed analysis of the structure of our rate of 
metastability might actually reveal -- in line with Lemma 4 in 
{\rm \cite{Par70}} --
that the remaining type-1 recursions (to define $N^*,T^*$ and $U^*$) 
are applied to type-2 functionals which are so simple (w.r.t. their 
dependence on the function argument) 
that our bound could be defined already in $T_0.$ We have to leave 
this for future research.
\end{remark}

\begin{remark}
In the special case where the mapping is nonexpansive and has a fixed point, 
we may trivially remove the boundedness condition as follows: 
let $G \subseteq X$ a closed, convex, nonempty subset. Let $U: G \to G$ be nonexpansive with a fixed point $p$ and $x \in G$. Let $b \in \N^*$ be such that $\|x - p\| \leq b/2$ and $\|p\| \leq b/2$. For all $t \in (0,1)$
put $x_t$ to be the unique point in $G$ such that 
$x_t = tUx_t + (1-t)x$. Let $(t_n)$, $\alpha$, $\gamma$ be as before. Denote, for all $n$, $x_n:=x_{t_n}$.
Then, for all $\eps > 0$ and $g: \N \to \N$ there is an $N \leq \Theta_{b,\eta,\tau,\text{\rm id},\alpha,\gamma}(\eps,g)$ such that for all $m,n \in [N,N+g(N)]$,
$\|x_m-x_n\| \leq \eps.$ To see this,
put $C$ to be the intersection of $G$ with the closed ball centred on $p$ with radius $b/2$. Clearly $C$ is closed, convex and nonempty. Set $T$ to be $U$ restricted to $C$, whose image is by nonexpansiveness also in $C$. Clearly, the diameter of $C$ is bounded by $b$, all elements of $C$ are bounded by $b$ and $x \in C$, so we may apply Theorem~\ref{main-metastab}.
\end{remark}

We now argue that the quantitative metastability of the sequence $(x_{t_n})$ is indeed a finitization in the sense of Tao of the following theorem (which is a somewhat restricted form of the main result in \cite{Rei80}).

\begin{theorem}[\cite{Rei80}]\label{main-thm-sunny}
Let $X$ be a Banach space which is uniformly convex and uniformly smooth, $C \subseteq X$ a closed, convex, bounded, nonempty subset, $T: C \to C$ be a uniformly continuous pseudocontraction and $x \in C$. For all $t \in (0,1)$ put $x_t$ to be the unique point in $C$ such that $x_t = tTx_t + (1-t)x$. Then for all $(t_n) \subseteq (0,1)$ such that $\lim\limits_{n \to \infty} t_n = 1$ we have that $(x_{t_n})$ converges to a fixed point of $T$, which we denote by $Qx$. In addition, the map $Q: C \to Fix(T)$ thus defined is a sunny nonexpansive retraction (and therefore the unique such one).
\end{theorem}
For this we now show that the metastability of $(x_{t_n})$ implies in an 
elementary way the above theorem: using just logic (and quantifier-free choice 
from $\N$ to $\N$) the metastability of $(x_{t_n})$ implies that  $(x_{t_n})$ is Cauchy, and therefore, since $X$ is complete and $C$ is closed, it is convergent 
(in the context of reverse mathematics, the latter fact uses arithmetical 
comprehension -- in fact a single use of $\Pi^0_1$-comprehension -- to get a 
fast converging subsequence as required to obtain the actual limit). 
It is clear that the limit does not depend on the $(t_n)$, so we can unambiguously dub it $Qx$. For the rest of the proof, we fix a $(t_n)$ and denote, for all $n$, $x_n:=x_{t_n}$. That $Qx$ is a fixed point follows from the continuity of $T$ and the fact (proven already in the beginning off Section~\ref{section-limsup}) that
$$\lim_{n \to \infty} \|x_n - Tx_n\| = 0,$$ whose trivial proof we recall here:
$$\|x_n - Tx_n\| = \|t_nTx_n + (1-t_n)x - Tx_n\| = \|(1-t_n)(x-Tx_n)\| \leq (1-t_n)b.$$ 
If $x$ is already a fixed point, then clearly for all $n$, $x_n=x$ and therefore $Qx=x$. We have thus shown that $Q$ is a retraction. To show that $Q$ is sunny and nonexpansive, we seek to apply Proposition~\ref{char-sunny}. Let $p \in Fix(T)$. Then
\begin{align*}
x_n - p &= t_nTx_n+(1-t_n)x-p\\
&= t_n(Tx_n-p) + (1-t_n)(x-p) \\
&= t_n(Tx_n-Tp) + (1-t_n)(x-p).
\end{align*}
Now we reuse parts of the argument from Claim 3 in the last proof from Section~\ref{section-limsup}. We have that
\begin{align*}
\|x_n-p\|^2 &= t_n \langle Tx_n - Tp, j(x_n-p) \rangle + (1-t_n)\langle x - p, j(x_n - p) \rangle \\
&\leq t_n \|x_n-p\|^2 + (1-t_n)\langle x - p, j(x_n - p)\rangle,
\end{align*}
so
$$\|x_n-p\|^2 \leq \langle x - p, j(x_n - p)\rangle$$
and therefore, using that $j$ is homogeneous,
$$\langle x - x_n, j(p-x_n) \rangle \leq 0.$$
By passing to the limit, using the continuity of $j$, we get that
$$\langle x - Qx, j(p-Qx) \rangle \leq 0,$$
which is what we needed to show.
\begin{remark}[for logicians; we use the terminology from \cite{Koh08}]\label{rem.6.5} \rm
As mentioned already, the proof of the metastability of $(x_{t_n})$ in 
Section \ref{section-approx} shows that it can be carried out in the formal system WE-PA$^{\omega}[X,\|\cdot\|,\eta,J_X,\omega_X,C].$ Note that the noneffective 
definition of the function $s_{p,g}$ can easily be avoided by using suitable 
rational approximations of $\| x_{n+g(n)}-p\|$ and $\| x_n-p\|.$ 
From this, the proof above of the convergence of $(x_{t_n})$
only requires classical logic, a fixed (in the parameters $T$, $x$, $(t_n)$ needed 
to define $(\| x_{t_n}\|)$) sequence  QF-AC$^{0,0}_-$ of 
instances of QF-AC$^{0,0}$ 
(in the terminology of reverse mathematics: $\Delta^0_1$-CA) 
and (a single use of) 
$\Pi^0_1$-CA. Both QF-AC$^{0,0}$ and $\Pi^0_1$-CA can (with classical logic) 
be combined into $\Pi^0_1$-AC$^{0,0}.$ 

Let us now specify 
the amount of classical logic needed when using the intuitionistically unproblematic principle AC$^{0,0}.$ By applying negative translation to the 
above proof of the Cauchyness of $(x_{t_n})$ we obtain in 
WE-HA$^{\omega}[X,\|\cdot\|,\eta,J_X,\omega_X,C]+$QF-AC$^{0,0}_-+$M$^0_-$ 
$$\forall k\in\N\neg\neg \exists n\in\N\,\forall m,\tilde{m}\ge n\ (\| x_{t_m}-x_{t_{\tilde{m}}}\|\le 1/(k+1))$$ (here M$^0$ denotes the Markov principle for 
numbers).
Hence,
$$\Sigma^0_2{\rm \mbox-DNE}: \ \neg\neg \exists n\in\N\,\forall m\in\N \,A_{qf}(n,m)\to 
 \exists n\in\N\,\forall m\in\N \,A_{qf}(n,m) \ \ \ (A_{qf} \ \mbox{quantifier-free})$$ (which also covers M$^0$) 
suffices. Using the closure of  WE-HA$^{\omega}[X,\|\cdot\|,\eta,J_X,\omega_X,C]+$AC$^{0,0}+\Sigma^0_1$-LEM 
under the rule of $\Sigma^0_2$-DNE (proven similarly as in \cite[Section 3]{KohSaf14}) one can conclude that 
even (a fixed -- in the parameters mentioned -- sequence $\Sigma^0_1$-LEM$_-$ 
of instances of) 
$$\Sigma^0_1\mbox{-LEM}: \ \exists n\in \N\,A_{qf}(n) \vee \neg \exists n\in\N\,A_{qf}(n)$$ 
suffices (when added to  WE-HA$^{\omega}[X,\|\cdot\|,\eta,J_X,\omega_X,C]+$AC$^{0,0}_-$) to prove the Cauchyness and -- in turn -- the  
convergence of $(x_{t_n})$ and the variational inequality (characterizing sunny nonexpansive retractions) from 
Proposition~\ref{char-sunny}.
\end{remark}

\section{Applications}\label{sec:apps}

The convergence of the resolvents, which form an implicit iteration schema, plays a role in proving the strong convergence of some explicit iteration schemas designed to compute fixed points of some nonlinear operators.

One such schema is the Halpern iteration \cite{Hal67}. If $T: C \to C$ is a mapping, $x$, $u \in C$ and $(\lambda_n) \subseteq (0,1)$, the Halpern iteration corresponding to this data is the sequence $(x_n)$, defined by:
$$x_0:=x, \quad x_{n+1}:=\lambda_{n+1} u+(1-\lambda_{n+1})Tx_n.$$

The convergence of this sequence for nonexpansive self-mappings of closed convex bounded nonempty subsets $C$ of uniformly smooth Banach spaces was obtained by Shioji and Takahashi \cite{ShiTak97} under Wittmann's \cite{Witt92} conditions on $(\lambda_n)$ and analyzed from the point of view of proof mining by the first author and Leu\c stean \cite{KohLeu12}, modulo the resolvent convergence. We are now in a position to complete this analysis under the additional hypothesis that $X$ is uniformly convex.

\begin{theorem}[{cf. \cite[Theorem 3.2]{KohLeu12}}]
Let $X$ be a Banach space which is uniformly convex with modulus $\eta$ and uniformly smooth with modulus $\tau$. Let $C \subseteq X$ a closed, convex, nonempty subset. Let $b \in \N^*$ be such that for all $x \in C$, $\|x\| \leq b$ and the diameter of $C$ is bounded by $b$. Let $T: C \to C$ be a nonexpansive mapping and $x$, $u \in C$. Put $\theta:=\id_\N$ and put $\alpha$ and $\gamma$ to be the functions defined, for all $n$, by $\alpha(n):=n$ and $\gamma(n):=n+1$. Let $(\lambda_n) \subseteq (0,1)$ be such that:
\begin{itemize}
\item $\sum_{n=0}^\infty \lambda_n = \infty$ with rate of divergence $\beta_1$;
\item $\lim_{n \to \infty} \lambda_n = 0$ with rate of convergence $\beta_2$;
\item $\sum_{n=0}^\infty |\lambda_{n+1}-\lambda_n| < \infty$ with Cauchy modulus $\beta_3$.
\end{itemize}
Denote by $(x_n)$ the Halpern iteration corresponding to this data. Let $\Sigma$ be defined by {\rm \cite[Theorem 3.2]{KohLeu12}}. Then, for all $\eps \in (0,2)$ and $g: \N \to \N$ there is an $N \leq \Sigma(\eps,\omega_\tau,g,b,\Theta_{b,\eta,\tau,\theta,\alpha,\gamma},\beta_1,\beta_2,\beta_3)$ such that for all $m,n \in [N,N+g(N)]$,
$\|x_m-x_n\| \leq \eps.$
\end{theorem}

An explicit iteration schema that is in addition amenable to pseudocontractions is the Bruck iteration \cite{Bru74A}. If $T: C \to C$ is a mapping, $x \in C$ and $(\lambda_n)$, $(\theta_n) \subseteq (0,1)$ such that for all $n$, $\lambda_n(1+\theta_n) \leq 1$, the Bruck iteration corresponding to this data is the sequence $(x_n)$, defined by:
$$x_1:=x, \quad x_{n+1}:=(1-\lambda_n) x_n + \lambda_n Tx_n - \lambda_n \theta_n (x_n -x).$$

The convergence of this sequence in some general framework containing the case of Lipschitzian pseudocontractive self-mappings of closed convex bounded nonempty subsets $C$ of uniformly convex and smooth Banach spaces was obtained by Chidume and Zegeye \cite{Chi} under some conditions on $(\lambda_n)$ and $(\theta_n)$ and then analyzed from the point of view of proof mining by K\"ornlein and the first author \cite{KorKoh14}, again modulo the resolvent convergence. We now complete their analysis.

\begin{theorem}[{cf. \cite[Corollary 2.10]{KorKoh14}}]
Let $X$ be a Banach space which is uniformly convex with modulus $\eta$ and uniformly smooth with modulus $\tau$. Let $C \subseteq X$ a closed, convex, nonempty subset. Let $b \in \N^*$ be such that for all $x \in C$, $\|x\| \leq b$ and the diameter of $C$ is bounded by $b$.  Let $T: C \to C$ be a Lipschitzian pseudocontraction of constant $L$ and $x \in C$. Let $(\lambda_n)$, $(\theta_n) \subseteq (0,1)$ satisfy the Chidume-Zegeye conditions.
Denote by $(x_n)$ the Bruck iteration corresponding to this data. Let $\chi$, $h$, $g'$ and $\Psi$ be defined as in {\rm \cite{KorKoh14}}. Put $\theta$ to be multiplication by $L$ and for all $n$, $\gamma(n):=h(n)+1$. Then, for all $\eps \in (0,2)$ and $g: \N \to \N$ there is an $N \leq \chi^M\left(\Theta_{b,\eta,\tau,\theta,\chi,\gamma}\left(\frac\eps2,g'\right)\right) + \Psi(\eps) + 1$ such that for all $m,n \in [N,N+g(N)]$,
$\|x_m-x_n\| \leq \eps$ and for all $\l \geq N$, $\|x_l-Tx_l\|\leq \eps$.
\end{theorem}

\begin{proof}
The only issue that needs additional justification is that $\chi$ and $\gamma$ are the required moduli for the auxiliary sequence $t_n = 1/(1+\theta_n)$ used in the original relative metastability proof. This is shown in the first lines of the proof of \cite[Theorem 2.8]{KorKoh14}.
\end{proof}

Another application of the rate of metastability extracted in this paper is given in \cite{Koh19}, where it is used to construct a rate 
of metastability for the strongly convergent Halpern-type Proximal Point Algorithm in uniformly convex and uniformly smooth Banach spaces 
from \cite{Aoyama}.
 
\section{Acknowledgements}

The authors have been supported by the German Science Foundation (DFG Project KO 1737/6-1).

\end{document}